\documentclass[twoside,10pt]{article}
\usepackage[utf8]{inputenc}
\usepackage[T1]{fontenc}
\usepackage{times}
\usepackage{amsmath,amsfonts,amssymb,amsthm,euscript,pifont}
\usepackage{mathrsfs}
\usepackage{color}
\usepackage[colorlinks=true]{hyperref}
\hypersetup{
    linkcolor = {blue},
    citecolor = {blue}
    }
\usepackage{xspace}

\usepackage{soul}
\usepackage{fullpage}
\usepackage{bbm}
\usepackage{authblk}
\usepackage[colorlinks=true]{hyperref}

\usepackage{dsfont}
\newcommand{\ind}{\mathds{1}}

\newcommand{\trace}{\textrm{Trace}}

\newtheorem{theorem}{Theorem}[section]
\newtheorem{lemma}[theorem]{Lemma}
\newtheorem{proposition}[theorem]{Proposition}
\newtheorem{corollary}[theorem]{Corollary}

\newtheorem{hypothesis}[theorem]{Hypothesis}

\newcommand{\er}{\mathbb{R}}

\newcommand{\Ee}{\displaystyle {\mathbb E}}

\newcommand{\Pp}{{\mathbb P}}

\newcommand{\FF}{{\mathcal F}}

\usepackage{upgreek}

\newcommand{\law}{\textit{{Law}}}
\newcommand{\loc}{\textit{loc}}

\newcommand{\LH}{L^2((0,T);H^1(\er^d))}

\newcommand\1{\mathds{1}}

\newcommand{\Bb}{\mathcal{B}}
\newcommand{\Cc}{\mathcal C}
\newcommand{\Dd}{\mathcal D}
\newcommand{\Ff}{\mathcal F}

\newcommand{\Ii}{{\mathcal I}}

\newcommand{\Ll}{\mathcal L}

\newcommand{\Nn}{{\mathcal N}}

\newcommand{\Ww}{\mathcal W}
\newcommand{\Xx}{\mathcal X}

\newcommand{\EE}{\mathbb E}
\newcommand{\EEE}{\mathcal E}

\newcommand{\PP}{\mathbb P}
\newcommand{\QQ}{\mathbb Q}

\newcommand{\Ep}{\mathbb{E}_{\mathbb P}}
\newcommand{\Eq}{\mathbb{E}_{\mathbb Q}}
\newcommand{\Ez}{\mathbb{E}_{\widetilde{\mathbb Q}}}

\newcommand{\Bmod}{\widehat{B}}

\newcommand{\hypmodo}{{\textit(${\textit A}_0$})\xspace}
\newcommand{\hypmodi}{{\textit(${\textit A}_1$})\xspace}
\newcommand{\hypmodii}{{\textit(${\textit A}_2$})\xspace}
\newcommand{\hypmodiiweaken}{{\textit(${\textit A}_2$-weaken{ed}})\xspace}

\newcommand{\hypcondo}{{\bf(H0)}\xspace}
\newcommand{\hypcondi}{{\bf(H1)}\xspace}
\newcommand{\hypcondii}{{\bf(H2)}\xspace}
\newcommand{\hypcondiii}{{\bf(H3)}\xspace}
\newcommand{\hypcondquatre}{{\bf(H4)}\xspace}
\newcommand{\hypcondcinq}{{\bf(H5)}\xspace}
\newcommand{\hypcondsix}{{\bf(H6)}\xspace}

\numberwithin{equation}{section}

\title{On the wellposedness of some McKean models with moderated or singular diffusion coefficient}
\author[1]{Mireille Bossy}
\author[2]{Jean-Fran\c{c}ois Jabir\footnote{The second author has been supported by the Russian Academic Excellence Project "$5$-$100$".}}
\affil[1]{Universit{\'e} C{\^o}te d’Azur, Inria, France}
\affil[2]{Department of Statistics and Data Analysis, Higher School of Economics, Moscow, Russia}

\date{\today}

\begin{document}
\maketitle
\begin{abstract}
We investigate the well-posedness problem related to two models of nonlinear McKean Stochastic Differential Equations with some local interaction in the diffusion term. First,  we revisit the case of the McKean-Vlasov dynamics with moderate interaction, previously studied by M{\'e}l{\'e}ard and Jourdain in \cite{JouMel-98}, under slightly weaker assumptions, by showing the existence and uniqueness of a weak solution using a Sobolev regularity framework instead of {a} H{\"o}lder one.
Second, we study the construction of a Lagrangian Stochastic model {endowed with a} conditional McKean diffusion term  in the velocity dynamics and a {nondegenerate} diffusion term  {in the position dynamics.}
\end{abstract}


\section{Introduction}

In this paper, we are  interested  in the wellposedness problem  of some singular nonlinear McKean SDEs in the McKean-Vlasov sense in $\er^d$,  in the particular situation where the diffusion term carries the singular McKean nonlinear dependency. General form of nonlinear McKean SDEs {is given by}
\begin{align}\label{eq:prototype}
d X_t &= b(X_t, \law(X_t)) dt + \sigma(X_t,\law(X_t)) dW_t, \quad \mu_0 = \law(X_0),
\end{align}
where $(W_t;\,t\geq 0)$ is a $\er^d$-valued standard Brownian motion, independent of $X_0${,} and $\mu_0$ is a given probability measure on $\er^d$.

The class of singular McKean models we want to consider here are  models for which the corresponding interacting particle system approximation (at least given formally, when the mean field limit is not yet established) {gives rise to} some singularity in the  kernel function. In the context of the model \eqref{eq:prototype}, this means that for any $N\geq 2$, for any family of mollifiers $g_{\varepsilon}:=\varepsilon^{-d}g(\frac{x}{\varepsilon})$ with parameter $\varepsilon>0$,  the mapping
\begin{align*}
x=(x^0,x^1,\ldots,x^N) \in \er^{(N+1)\times d}\mapsto b \text{ or } \sigma\Big(x^0, \Big(g_\varepsilon * \mu_x^N\Big)\Big) \in \er^d, \text{ with  } \mu_x^N=\frac{1}{N}\sum_{i=1}^N \delta_{x^i}
\end{align*}
attains its maximum {norm} on the subset  $\bigcup_{i=1}^N\{x^0=x^i\}$ which tends  to $+\infty$ as $\varepsilon$ tends to zero. Well known examples of McKean SDEs with  singular drift kernel $b$ are those of the probabilistic interpretation of Burgers' equation (Sznitman \cite{Sznitman-86}), the stochastic vortex method model for fluid flow  (see e.g. Chorin \cite{chorin-73}, M{\'e}l{\'e}ard \cite{Meleard-01} among others), or more recently the probabilistic interpretation of the Keller Segel equation for chemotaxis modeling (see e.g. Fournier and Jourdain \cite{FouJou-17}).

More precisely, we are interested  in the wellposedness of the following coupled processes $(X_t,Y_t;\,t\geq 0)$ on a probability space $(\Omega, \FF, \Pp)$ satisfying the conditional McKean SDE
\begin{equation}\label{eq:NonlinearLangevin}
\left\{
\begin{aligned}
X_t & = X_0 + \int_0^t b(X_s, Y_s) ds + \int_0^t \sigma(X_s) dB_s, \\
Y_t & = Y_0 + \int_0^t \Ep[ \ell(Y_s) | X_s]  ds + \int_0^t \Ep[\gamma(Y_s) | X_s] dW_s.
\end{aligned}
\right.
\end{equation}
The initial condition $(X_0,Y_0)$ is distributed according to a given initial law $\mu_0$, $(W_t;\,t\geq 0)$ and $(B_t;\,t\geq 0)$ are two independent $\er^d$ standard Brownian motions,
$x \mapsto\sigma(x)$  is  a $\er^d\times\er^d$ valued function. Before briefly describing our hypotheses on the coefficients of \eqref{eq:NonlinearLangevin_intro}, let us make some comments on such models.

Our particular interest for the study of singular dynamics as \eqref{eq:NonlinearLangevin} is motivated by the wellposedness problem related to the class of Lagrangian stochastic models for turbulent flows. This class of models have been introduced in the general framework of the statistical description of fluid motions and
aimed to describe the main characteristic properties (position, velocity, ...) of a generic particle of a particular fluid flow. From the turbulent modeling viewpoint, such SDEs are known as Lagrangian "fluid-particle" models and are translation in a Lagrangian point of view (SDE) of some Eulerian PDE turbulence models (see e.g. Pope \cite{pope-94}, \cite{pope-11}, Durbin and Speziale \cite{DurSpe-94}). These models involve a particular family of nonlinear McKean-Vlasov SDEs where the McKean nonlinearities are of conditional form. Such particular form of nonlinearity models the influence of the macroscopic components of the flow on the particle motion. In some of our recent works, \cite{bossy-11}, \cite{bossy-15}, we have studied toy-version  models of conditional McKean SDEs where the singularity is concentrated in the drift term. From  a mathematical viewpoint, the wellposedness results obtained in \cite{bossy-11},\cite{bossy-15} are still far from covering the complexity of a  meaningful 'fluid-particle' model, as such Lagrangian models contain conditional McKean nonlinearity in both drift and diffusion components. In this paper, we focus on
singular McKean diffusive characteristic that motivates our interest in new wellposedness results in that direction.

In \cite{bossy-11},\cite{bossy-15} and for the construction of (numerical) approximation (we refer to Bossy {\it et al.} \cite{bossy-16},\cite{bossy-17} for some numerical description cases and experiments), we analyze the SDE \eqref{eq:NonlinearLangevin} in the framework of an apriori existing density $\rho_t(x,y) dx dy   = \Pp(X_t\in dx,\,Y_t\in dy)$. The model \eqref{eq:NonlinearLangevin} thus  becomes
\begin{equation}\label{eq:NonlinearLangevin_intro}
\left\{
\begin{aligned}
X_t & = X_0 + \int_0^t b(X_s, Y_s) ds + \int_0^t \sigma(X_s) dB_s, \\
Y_t & = Y_0 + \int_0^t \Lambda[X_s;\rho_s] ds + \int_0^t \Gamma[X_s;\rho_s] dW_s,
\end{aligned}
\right.
\end{equation}
with  $\Lambda$ and $\Gamma$  defined, for $(x,f)$ in
$\er^d\times L^1(\er^d\times\er^d)$, as
\begin{align*}
&\Lambda[x;f]  =
\dfrac{\int_{\er^d} \ell(y) f(x,y) dy }{\int_{\er^d} f(x,y) dy }
 \ind_{\{\int_{\er^d} f(x,y) dy\neq 0\}}
\quad \mbox{and} \quad  \Gamma[x;f]  =
\dfrac{\int_{\er^d} \gamma(y) f(x,y) dy }{\int_{\er^d} f(x,y) dy }
 \ind_{\{\int_{\er^d} f(x,y) dy\neq 0\}}.
\end{align*}
In comparison, our wellposedness result for the solution of \eqref{eq:NonlinearLangevin}, presented in Section \ref{sec:ConditionalDiffusion}, uses a $L^2(\Omega)$-fixed point construction and a suitable Girsanov transformation that relies on the strong ellipticity assumption on $\sigma$. Essentially, our working hypotheses will be to assume boundedness {and} Lipschitz continuity of $b,\sigma,\ell$ and $\gamma$ for the wellposedness of a weak solution to \eqref{eq:NonlinearLangevin}, and some {$L^p$ density condition on the initial distribution and a} uniform elliptic property on $\gamma$ to handle pathwise uniqueness. {At all time $t$, the time-marginal distributions $\law(X_t,Y_t)$ of this strong solution further admit a density function $\rho_t$, and so our constructed solution to \eqref{eq:NonlinearLangevin} is also solution to \eqref{eq:NonlinearLangevin_intro}.}

In the context of complex flow modeling, we would like  to  emphasi{se} that  a  targeted form of \eqref{eq:NonlinearLangevin} is a coupled position-velocity $(X_t,U_t;\,0\leq t\leq T)$ kinetic process with {degenerate}  diffusion {in the $X$-component} {together with a linear drift  $b(x,y)=y$:}
\begin{align*}
\left\{
\begin{aligned}
X_t & = X_0 + \int_0^t U_s ds \\
U_t & = Y_0 + \int_0^t \Ep[ \ell(U_s) | X_s]  ds + \int_0^t \Ep[\gamma(U_s) | X_s] dW_s.
\end{aligned}
\right.
\end{align*}
But unbounded drift case, {degenerate}  diffusion and singular  McKean kernel are a mix{ture} of difficulties {that are} quite hard to overcome {jointly}. 

For future works, to overcome the strong ellipticity assumption on $\sigma$ in \eqref{eq:NonlinearLangevin}, we further investigate some weaker {characterisation} method based on mild-equation formulation as in \cite{bossy-15}. In Section \ref{sec:ModerateNonlinearDiffusion}, we present a step further in that direction, applying such technique for our second study case on moderated McKean local diffusion equation:
\begin{align}\label{eq:moderated-intro}
\left\{
\begin{aligned}
&X_t =  X_0 + \int_0^t \sigma ( u(s,X_s) ) dW_s, \quad 0\leq t \leq T,\\
&{d}\;\!\law(X_t)=u(t,x) dx\,\mbox{ with }\,u \in {L^\infty((0,T)\times\er^d)\cap~ }L^2((0,T){\times}\er^d),\\
&u(0,{x})=u_0{(x),\,x\in\er^d,}\text{{ where $u_0$ is a given probability density function on $\er^d$}},
\end{aligned}
\right.
\end{align}
for any arbitrary time horizon $0<T<\infty$.
Nevertheless, our existence proof based on approximation method needs some strict monotonicity assumption which still coincides with the strong ellipticity in the one dimensional  framework.
In \cite{JouMel-98}, Jourdain and M{\'e}l{\'e}ard studied a moderately interacting model such as \eqref{eq:moderated-intro}, extending a previous work from Oelschl{\"a}ger \cite{Oelschlager-85}  on a moderately interacting model, where both the drift and diffusion coefficients depend locally on the time marginal densities of the {law of the solution} 
that are supposed to be smooth enough.  Whenever the nonlinearity is reduced to the diffusion part, the model in \cite{JouMel-98} {reduces to}:
\begin{equation}\label{eq:JourMel-98}
\left\{
\begin{aligned}
&X_t =  X_0 + \int_0^t \sigma ( p(s,X_s) ) dW_s, \quad 0\leq t \leq T,\\
&{d}\;\!\law(X_t)=p(t,x)\,dx\,\mbox{with}\,p\in \Cc^{1,2}_b([0,T]\times\er^d),\\
& p(0,\cdot)\mbox{ is a given probability density that belongs in the H\"older space $H^{2+\alpha}(\er^d)$ with $0<\alpha<1$}.
\end{aligned}
\right.
\end{equation}
In Section \ref{sec:ModerateNonlinearDiffusion}, we prove the wellposdness of a strong solution of \eqref{eq:moderated-intro}, mainly replacing the condition $p\in\Cc^{1,2}_b([0,T]\times\er^d)$ by $p\in L^{\infty}((0,T)\times\er^d)\cap L^{2}((0,T)\times\er^d)$  (replacing the strong ellipticity condition on  $\sigma$ needed for the equation  \eqref{eq:JourMel-98} by a strict monotonicity condition). Our proof {is focused} 
on the simple case where the diffusion component is given by $\sigma(r)I_d$ for a scalar function $\sigma:[0,\infty)\rightarrow[0,\infty)$. Extensions to further multidimensional diffusion component are discussed at the end of Section \ref{sec:ModerateNonlinearDiffusion}.

Our main results are Theorem \ref{thm:moderatedexist} in Section \ref{sec:ModerateNonlinearDiffusion} which states the strong wellposedness for the moderated McKean local diffusion equation \eqref{eq:moderated-intro}, and Theorem \ref{thm:wellposedNonlinearLangevin} in Section \ref{sec:ConditionalDiffusion} for the strong wellposedness of condition{al} McKean SDEs \eqref{eq:NonlinearLangevin_intro}. In the two cases, we obtained weak uniqueness of the solution with slight{y} weaker condition{s}.

We end this introductory section with a short review of results and approaches from the literature  for SDEs with McKean diffusion term, in order give some insights  to the reader with the two particular cases that we are addressing in this paper.

\subsection*{Review of some wellposedness results for nonlinear SDEs with McKean diffusion term}

We consider McKean-Vlasov SDEs of the following specific form in $\er^d$,
\begin{equation}\label{eq:ProtoDiffMcKean}
 X_t  = X_0 + \int_0^t \sigma( X_s, \law(X_s)) dW_s,\,0\leq t\leq T,
\end{equation}
up to a (possibly infinite) horizon time $T$.

Under the assumptions that $\EE[|X_0|^p]<\infty$, $1\leq p<\infty$,  and $\sigma$ is continuous on $\er^d\times {\mathcal{P}_p}(\er^d)$ for ${\mathcal{P}_p}(\er^d)$ being the space of probability measures with $p${-th} finite moments,  Funaki \cite{Funaki-84} showed the existence, on any arbitrary time interval, of a weak solution to \eqref{eq:ProtoDiffMcKean} in terms of a martingale problem. Uniqueness of the solution to the martingale problem holds under the assumption that
\[\| \sigma (x,\mu) - \sigma(y,\nu) \| \leq  C |x-y| + \kappa( {\Ww_p}(\mu,\nu)),\]
where {$\Ww_p$} is the Wasserstein distance endowed with the cost function $|x-y|^p$, and $\kappa:[0,\infty)\rightarrow[0,\infty)$ is a strictly increasing function such that $\kappa(0)=0$ and $\lim_{\epsilon\rightarrow 0^+}\int_\epsilon^\infty 1/\kappa^2(\sqrt{r})\,dr=\infty$.
Oelschl\"ager \cite{Oelschlager-84} considered the analog situation where $\sigma$ {is} bounded and Lipschitz for the metric
  \[
  \Vert\mu-\nu\Vert=\sup\left\{\int f(x)\left(\mu(dx)-\nu(dx)\right)\,{:}\,\max_{x\in\er^d}|f(x)|\leq 1\,\mbox{and}\,\Vert f\Vert_{Lip}:=\sup_{x\neq y}\frac{|f(x)-f(y)|}{|x-y|}\leq 1\right\},
  \]
  and proved the existence of a solution in law, as well as a weak propagation of chaos result for the related stochastic particle system. Both cases include the particular situation when the interaction kernel has the form: $\sigma(x,\mu) = \int_{\er^d} \sigma(x,y) \mu(dy)$. Moreover, in this framework, M\'el\'eard \cite{Meleard-96} showed, through a fixed point argument in the space $(\mathcal{P}_2(\Cc([0,T];\er^d)),W_2)$,  that whenever, $\sigma:\er^d\times\er^d\rightarrow\er^{d\times d}$ is Lipschitz continuous w.r.t. the two variables, {the} pathwise wellposedness and strong-pathwise propagation of chaos {holds} for the related stochastic particle system.

Jourdain and M\'el\'eard \cite{JouMel-98} extended the work of Oelschl{\"a}ger \cite{Oelschlager-85} on the  moderately interacting drift term model and prove the wellposedness of  \eqref{eq:JourMel-98}
with the following assumptions:
\begin{itemize}
\item[$\circ$]  {$p(0,x)=p_0(x)$ where} $p_0$  belongs to the H\"older space $H^{2+\alpha}(\er^d)$ with $0<\alpha<1$;
\item[$\circ$] $\sigma{:r\in \er\mapsto \sigma(r)\in\er^{d\times d}}$ is a  Lipschitz function, $\Cc^3$ on $\er$, with values in the space of symmetric non-negative matrices $d\times d$;
\item[$\circ$]   Strong ellipticity holds  for $\sigma$: {there exists $m_\sigma>0$ such that}
$\forall x\in\er^d,\forall {r}\in\er,\quad x^* \sigma({r}) x \geq m_\sigma |x|^2$;

\item[$\circ$] Non negativity holds for the diffusion matrix leading to the Fokker-Planck equation  written on divergence form:
\begin{align*}
\forall x\in\er^d,\forall {r}\in\er, \quad x^* \left( (\sigma\sigma^*)'({r}) {r} + (\sigma\sigma^*) ({r}) \right)  x \geq 0;
\end{align*}
(This latest assumption is used to derive the uniqueness from the Fokker-Planck equation related to \eqref{eq:JourMel-98}, written in divergence form:
\begin{equation}\label{eq:pdeJourMel-98}
\left\{\begin{aligned}
&\frac{\partial p}{\partial t} = \sum_{i=1}^d \frac{\partial}{ \partial x_i}\left( \frac{1}{2} \sum_{j=1}^d \Big( (\sigma\sigma^*)^\prime_{ij}(p) p + (\sigma\sigma^*)_{ij}(p)\Big) \frac{\partial  p}{\partial x_j}\right)~\text{on}~  (0,T)\times \er^d,\\
&p(0,x) = p_0(x),~{x\in\er^d,}
\end{aligned}\right.
\end{equation}
from maximum principle argument.)
\item[$\circ$] Strong ellipticity holds on the leading matrix: there exists $m_{\text{div}}>0$ such that
\[\forall x\in\er^d,\forall {r}\in\er, \quad x^* \left( (\sigma\sigma^*)'({r}) {r} + (\sigma\sigma^*) ({r}) \right)  x  \geq m_{\text{div}} |x|^2.\]
({With this additional} 
assumption, the Cauchy problem \eqref{eq:pdeJourMel-98} has a solution in $H^{1 + \frac{\alpha}{2},2+\alpha}(\er^d)$, and the nonlinear SDE admits a unique strong solution.)
\end{itemize}

Kohatsu-Higa and Ogawa  in \cite{KohOga-97} considered  nonlinear McKean-Vlasov dynamic in convolution form
\begin{align}\label{eq:intro_kohatsu}
X_t  = X_0 + \int_0^t A( X_s, \sigma* \law(X_s) ) dW_s.
\end{align}
Assuming that $A$ and $\sigma$ are Lipschitz with at most  linear growth, they prove the wellposedness of a strong solution and particle-time discrete approximation.

Jourdain and Reygner in \cite{JouRey-13} considered particular cases of scalar equation related to, and around, porus media equation which correspond to the case of equation \eqref{eq:intro_kohatsu} with $\sigma(x)= \ind_{\{x\geq 0\}}$ and $A(x,u) = A(u) >0$. The case $A(u)\geq 0$ is also studied using {t}he limit of a reordered particle system. 

Recently, Mishura and Veretennikov in \cite{MisVer-17} consider a model of the form
$$
dX_t=\sigma[t,X_t;\law(X_t)]\,dW_t,\,X_0\sim\mu_0,
$$
where
$$
\sigma[t,x;\mu]=\int \sigma(t,x,y)\mu(dy)
$$
for  $\sigma:(0,\infty)\times\er^d\times\er^d\rightarrow \er^{d\times d}$.
Assuming that $X_0$ has finite fourth order moments, $(x,y)\mapsto \sigma(t,x,y)$ has linear growth, uniformly in $t$, and $(t,x,y)\mapsto \sigma(t,x,y)\sigma^*(t,x,y)$ is uniformly strongly elliptic, the SDE admits at least one weak solution.

We end this review {by} mentioning some recent works in the direction of the wellposedness of the following system of SDE
\begin{align*}
\left\{
\begin{aligned}
&\dfrac{dS_t}{S_t}   =  r dt +  \frac{a(Y_t)}{\sqrt{\Ee[a^2(Y_t)|S_t]}} \sigma_{\text{Dup}} (t,S_t) S_t dW_t,    \\
&dY_t = \alpha(t,Y_t) dB_t  + \xi(t) dt.
\end{aligned}
\right.
\end{align*}
Such models arise in mathematical finance for the calibration of local and stochastic volatility  models, and where $\sigma_{\text{Dup}}(t,y)$ is the Dupire's local volatility function (see Gyongy \cite{Gyongy-86}). We emphasis{e} that a major difference with the model \eqref{eq:NonlinearLangevin} is in the conditioning variable  which is the coupled variable $X_t$ in our case and the unknown $S_t$ in the volatility calibration model. This particular case  generates different and yet hard difficulties {compared to \eqref{eq:NonlinearLangevin}}.
The existence of a local-in-time solution to the Fokker-Planck equation associated to this model has been established by Abergel and Tachet in \cite{AbeTac-11}, while Jourdain and Zhou\cite{JouZho-17} recently obtained a first global-in-time wellposedness result in the case when $Y$ is a (constant in time) discrete valued random variable.

\subsubsection*{Some notations}

Hereafter, $\Cc$ denotes the space of continuous functions equipped with the uniform norm $\Vert f\Vert_\infty=\max|f|$ and $\Cc^{k}$ denotes  the space of  $k$ times {continuously} differentiable functions. $\Cc_c$ and $\Cc^{k}_c$  respectively refer to the corresponding compactly supported subsets. For $m\geq 1$, and $1\leq p\leq \infty$, $L^p(\er^m)$  denotes the Lebesgue space of all Borel (measurable) functions $f:\er^m\rightarrow\er$ such that $\Vert f\Vert^p_{L^p(\er^m)}:=\int |f(z)|^p\,dz<\infty$, and $W^{1,p}(\er^m)$  {denotes} the Sobolev space
\[W^{1,p}(\er^m)=\{f\in L^p(\er^m)\text{ s.t. }\Vert \nabla_z f\Vert_{L^p(\er^m)}<\infty\},\]
 equipped with the norm $\Vert f\Vert_{W^{1,p}(\er^m)}:=\Vert f\Vert_{L^{p}(\er^m)}+\Vert \nabla_z f\Vert_{L^{p}(\er^m)}$. As usual, $H^1(\er^m)$ stands for the particular space $W^{1,2}(\er^m)$. For  $1\leq q\leq \infty$, $L^{q}((0,T);W^{1,p}(\er^m))$ denotes the space of Borel functions $f$ defined on $(0,T)\times\er^m$ such that
$$
\Vert f\Vert_{L^{q}((0,T);W^{1,p}(\er^m))}:=\Vert f\Vert_{L^{q}((0,T);L^{1,p}(\er^m))}+\Vert \nabla f\Vert_{L^{q}((0,T);L^{1,p}(\er^m))}
$$
is finite. Finally, the index $\loc$ will refer to local integrability property, namely $f$ belongs in $L^p_{\loc}$ if for all $0\leq R<\infty$, $f\ind_{\{B(0,R)\}}$ is in $L^p(\er^m)$.

\section{The  moderated McKean local diffusion equation revisited}\label{sec:ModerateNonlinearDiffusion}

In this section, we consider the wellposedness problem, up to an arbitrary finite horizon time $T>0$, for the following SDE:
\begin{subequations}
\label{eq:moderated}
\begin{align}
&X_t =  X_0 + \int_0^t \sigma ( u(s,X_s) ) dW_s, \quad 0\leq t \leq T,\label{eq:moderatedSDE}\\
&{d}\;\!\law(X_t)=u(t,x) dx\,\mbox{ where }\,u\mbox{ belongs in }  {L^\infty((0,T)\times\er^d)\cap~ }L^2((0,T){\times}\er^d),\label{eq:moderatedlaw}\\
&u({0,x})=u_0{(x),\,x\in\er^d,}\,\text{ where }\,u_0\,\text{is a given probability density function}\label{eq:moderatedinit}.
\end{align}
\end{subequations}
For the sake of simplicity, from now on, we restrict ourselves to the case of a diffusion matrix $\sigma$ mainly diagonal; that is $\sigma(r)=\sigma(r)I_d$ for $\sigma:\er\rightarrow\er$ and $I_d$ the identity $d\times d$ matrix. Extensions to more general diffusion matrices will be discussed at the end of this section (see Subsection \ref{rmk:ModerateGeneral}).

Let us further point out that the notion of solution to \eqref{eq:moderated} is intentionally restricted to the class of solutions satisfying \eqref{eq:moderatedlaw}. We consider the class of solutions of continuous processes satisfying \eqref{eq:moderatedSDE} and whose time-marginal distributions admit a representant in $L^2(\er^d)$ for a.e. $0\leq t\leq T$, {and in $L^\infty((0,T)\times\er^d)$}. The choice of working with this particular class is mainly motivated by the use of {comparison principles and } energy estimates techniques (see e.g. Evans \cite{Evans-97} and Vasquez \cite{Vasquez-06}) for the time-marginal distributions solution to the Fokker-Planck equation related to \eqref{eq:moderated}.  Energy estimates will  enable us to construct a suitable approximation to \eqref{eq:moderated} and also to deduce the uniqueness of the marginal distributions $u(t),\,0\leq t\leq T$.
We emphasis{e} that the divergence form for the Fokker-Planck equation makes appear {as a coefficient the map $r\mapsto \alpha(r)$ defined as}
\begin{align}\label{def:alpha}
\alpha(r)  := (\sigma^2(r) r)^\prime = 2\sigma^\prime (r) \sigma(r) r  +  \sigma^2(r),\text{ for }r\in\er^+,
\end{align}
 which our main hypothesis {is based on}.

{\paragraph{Remark:} {\it In the case when $\sigma$ and $\alpha$ are bounded, our proof arguments and subsequent wellposedness results can be extended to the class of solutions to \eqref{eq:moderatedSDE} satisfying $u\in L^2((0,T)\times\er^d)$ in place of \eqref{eq:moderatedlaw}.}}

Throughout this section, {Equation} \eqref{eq:moderated} is considered under the following set of assumptions:
\begin{hypothesis}\label{hypo:moderated}
\item{\bf\hypmodo} $u_0~\text{is a probability density function in}~L^1(\er^d) \cap  L^\infty(\er^d)$ such that $\int_{\er^d}|x|^2 u_0(x)dx < \infty$.
\item{\bf\hypmodi} The map $r \mapsto \sigma(r)\in\er$ is 
{continuously differentiable} on $\er^+$.
\item{\bf \hypmodii} The map $r\mapsto \alpha(r)$
is {continuous} on $\er^+$,  and there exists some constant $\eta>0$ such that
\[
\alpha(r)\geq \eta >0 ,\,\forall\,r\geq 0.
\]
\end{hypothesis}
Assumption \hypmodii  ensures that $\sigma$ is in $\Cc^1([0,\infty))$ and implies a classical assumption on the uniform positiv{ity} of $\sigma^{{2}}$:
$$
\sigma^2(r) \geq \eta,\,\forall\,r\geq 0,
$$
which further implies the uniform ellipticity of $\sigma^{2} I_d$. Yet, most of the time, we will also make use of the following assumption which{, together with a monotonic property of $\alpha$ (see Theorem \ref{thm:moderatedexist} and Proposition \ref{prop:moderateduniq} below),} allows  {possible} degeneracy of the diffusion $\sigma$ at point $r=0$:
\begin{hypothesis}
\item{\bf\hypmodiiweaken} The map $r\mapsto \alpha(r)$ is {continuous} on $\er^+$, and
\[ \alpha(r) \geq 0,\,\forall\,r\geq 0.\]
\end{hypothesis}

The main result of this section is the wellposedness for equation \eqref{eq:moderated}  given by the following theorem:
\begin{theorem}\label{thm:moderatedexist}
Under Hypothesis \ref{hypo:moderated}, there exists a unique strong solution to \eqref{eq:moderated}.
Uniqueness in law holds true under \hypmodo, \hypmodi,  \hypmodiiweaken  with the additional hypothesis that $r\mapsto \alpha(r)$ is strictly increasing.
\end{theorem}

\paragraph{Main ingredients and steps of the proof. } The rest of this section is devoted to our proof of Theorem \ref{thm:moderatedexist} that relies on the following three main ingredients:
\begin{itemize}
\item[1.] An appropriate form of $\varepsilon$-nondegenerat{e} approximation of the diffusion $\sigma$.
In a first step,  we show the wellposedness of a family of  $\varepsilon$-approximation $\{(u^\varepsilon(t),\,X^\varepsilon_t;\,0\leq t\leq T), \varepsilon>0\}$ to \eqref{eq:moderated} where $\sigma$ is replaced by $\sigma_\varepsilon$ defined by
\begin{equation}\label{epsilondiffusion}
\sigma^2_\varepsilon (r)  = \sigma^2(r) + \varepsilon,~\forall r \geq 0.
\end{equation}
Notice that \eqref{epsilondiffusion} produces a suitable approximation of the map $r\mapsto \alpha(r)$ by
\begin{equation}\label{eq:mode_nondegen_dif}
\alpha_\varepsilon (r) = (\sigma^2_\varepsilon (r)  r )^\prime  = (\sigma^2(r)r)^\prime  + \varepsilon = \alpha(r) + \varepsilon,~\forall r \geq 0.
\end{equation}
\item[2.]
The construction of $(u^\varepsilon(t),\,X^\varepsilon_t;\,0\leq t\leq T)$ is then obtained from a preliminary existence result of a $L^\infty((0,T)\times\er^d)\cap L^2((0,T);H^1(\er^d))$-weak solution to the related Fokker-Planck equation, and some uniform  energy estimates w.r.t $\varepsilon$ on this solution. Such estimates allow to deduce successively the relative compactness of $\{u^\varepsilon,\varepsilon>0\}$ in $L^2((0,T);H^1(\er^d))$ and of $\{\law(X^\varepsilon_t;\,0\leq t\leq T),\varepsilon>0\}$ in $\mathcal{P}(\Cc([0,T];\er^d))$. The corresponding limits of converging subsequences are then shown to jf{be} a solution of the martingale problem related to \eqref{eq:moderated}. This main step is stated in Proposition \ref{prop:limit_nonlinear_smooth_sdepde} below.

\item[3.] Uniqueness in law is obtained from a mild-form  equation in $L^2((0,T)\times\er^d)$ derived from the Fokker-Planck equation in Proposition \ref{prop:moderateduniq}.  The  mild approach used here allows us to get ri{d} of the strong ellipticity hypothesis for $\sigma$, at least at point $0$.
The {weak} uniqueness result is then obtained under \hypmodii, but also under \hypmodiiweaken with the adding of the strict monotonicity for $\alpha$. {Uniqueness in the pathwise sense is stated in Proposition \ref{prop:moderatedstronguniq}.}
\end{itemize}

\subsection{Nondegenerate approximation of \eqref{eq:moderated}}
In this section, we construct a solution to the SDE
\begin{subequations}\label{eq:mode_sde_smooth}
 \begin{align}
 &X^\varepsilon_t =  X_0 + \int_0^t \sqrt{\sigma^2_\varepsilon( u^\varepsilon(s,X_s) )} dW_s, \quad 0\leq t \leq T,\label{eq:moderatedapproxSDE}\\
&{d}\;\!\law(X^\varepsilon_t)=u^\varepsilon(t,x) dx, \mbox{ with } u^\varepsilon \in {L^\infty((0,T)\times\er^d)\cap~ }L^2((0,T){\times}\er^d),\label{eq:moderateapproxdlaw}\\
&u^\varepsilon({0,x})=u_0{(x)}{, x\in\er^d}\,\label{eq:moderateapporxdinit},
\end{align}
\end{subequations}
where $\sigma^2_\varepsilon (r)  = \sigma^2(r) + \varepsilon,~\forall r \geq 0$, and we {show} some appropriate density estimates for the marginal densities $u^\varepsilon(t)$.  Defined as such, the diffusion coefficient $\sigma_\varepsilon$ still satisfies \hypmodi. Our existence proof is mainly deduced from a PDE analysis of the smoothed Fokker-Planck equation related to \eqref{eq:mode_sde_smooth}:
\begin{equation}\label{eq:mode_pde_smooth}
\left\{
\begin{aligned}
&\frac{\partial u^\varepsilon}{\partial t } - \frac{1}{2}\triangle_x (\sigma_\varepsilon^2(u^\varepsilon) u^\varepsilon) = 0,{\text{ on}\,(0,T)\times\er^d,}\\
&u^\varepsilon({0},x)=u_0(x),\,x\in\er^d.
\end{aligned}
\right.
\end{equation}
We proceed for the existence of a solution to \eqref{eq:mode_pde_smooth}, first  by  exhibiting the existence of a weak solution of a linearized version of \eqref{eq:mode_pde_smooth}, and next by a fixed point argument we deduce the existence. These results are given in the {following two} lemmas: 
\begin{lemma}\label{lem:linear_pde}
Assume \hypmodo, \hypmodi and \hypmodiiweaken. Let $v=v(t,x)$ be a non-negative given function {belonging to} ${L^\infty((0,T)\times\er^d)\cap }L^2((0,T)\times\er^d)$. Then there exists $\overline{u}^\varepsilon\in\LH\cap \Cc([0,T];L^2(\er^d))$ such that, for all $f\in\Cc^{1,2}_c([0,T]\times\er^d)$, {for all $0\leq T_0\leq T$,}
\begin{equation}\label{eq:mode_pde_linear}
\begin{aligned}
&\int_{\er^d}\overline{u}^\varepsilon(T_0,x)f(T_0,x)\,dx-\int_{\er^d}u_0(x)f(0,x)\,dx\\
&\qquad=\int_{(0,T_0)\times\er^d}\overline{u}^\varepsilon(t,x)\partial_t f(t,x)-{\frac{1}{2}\alpha_\varepsilon}(v(t,x)) \nabla_x\overline{u}^\varepsilon(t,x)\cdot \nabla_x f(t,x) \,dt\,dx,
\end{aligned}
\end{equation}
for {${\alpha_\varepsilon}$ defined as in \eqref{eq:mode_nondegen_dif}}.
In addition, $\overline{u}^\varepsilon$ has nonnegative values a.e. on $(0,T)\times\er^d$,
\begin{equation}\label{eq:linear_ineq_Linfini}
\Vert \overline{u}^\varepsilon\Vert_{L^\infty((0,T)\times\er^d)}\leq \Vert u_0\Vert_{L^\infty(\er^d)},
\end{equation}
and
 \begin{align}\label{eq:linear_ineq_energy}
\max_{t\in[0,T]} \|\overline{u}^\varepsilon(t)\|^2_{L^2(\er^d)} + \varepsilon \int_0^T \|\nabla_x \overline{u}^\varepsilon(t) \|^2_{L^2(\er^d)} dt \leq \|u_0\|^2_{L^2(\er^d)}.
\end{align}
\end{lemma}
For the nonlinear PDE \eqref{eq:mode_pde_smooth}, we extend the notion of a $\LH$-weak solution, stated in {this lemma~\ref{lem:linear_pde}}, as a function $u^\varepsilon\in\LH$ such that: for all $0< T_0\leq T$, and $f\in\Cc^\infty_c([0,T_0]\times\er^d)$,
\begin{align*}
&\int_{\er^d}u^\varepsilon(T_0,x)f(T_0,x)\,dx-\int_{\er^d}u_0(x)f(0,x)\,dx\\
&\qquad=\int \int_{(0,T_0)\times\er^d}u^\varepsilon(t,x)\partial_t f(t,x) \,dt\,dx+\int_{(0,T_0)\times\er^d}\frac{1}{2}\sigma_\varepsilon^2(u^\varepsilon(t,x)) u^\varepsilon(t,x)\triangle_x f(t,x) \,dt\,dx.
\end{align*}
\begin{lemma}\label{lem:WellposednessSmoothpde}
Assuming that \hypmodo, \hypmodi and \hypmodiiweaken~{hold},  the nonlinear PDE \eqref{eq:mode_pde_smooth} admits a unique nonnegative ${\Cc([0,T];L^2(\er^d))\cap }L^2((0,T);H^1(\er^d))$-weak solution  
$u^\varepsilon$. This solution is uniformly bounded with
\begin{equation}\label{eq:smoothednonlinear_ineq_Linfini}
\|u^\varepsilon\|_{L^\infty((0,T)\times\er^d)}\leq \|u_0\|_{L^\infty(\er^d)},
\end{equation}
and satisfies the energy inequality:
\vspace{-0.3cm}
\begin{equation}\label{eq:smoothednonlinear_ineq_energy}
\sup_{t\in[0,T]} \|u^\varepsilon(t)\|^2_{L^2(\er^d)} +  \varepsilon \int_0^T \|\nabla_x u^\varepsilon(t) \|^2_{L^2(\er^d)} dt \leq \|u_0\|^2_{L^2(\er^d)}.\\
\end{equation}
In addition, we have, for all $0\leq T_0\leq T$,
\begin{equation}\label{eq:Vasqueztype_ineq}
 \int_{\er^d} \Psi_\varepsilon(u^\varepsilon(T_0,x))\,dx+\frac{1}{2}\int_{(0,T_0)\times\er^d} \left|\nabla_x \Phi_\varepsilon(u^\varepsilon(t,x))\right|^2\,dt\,dx=\int_{\er^d}\Psi_\varepsilon(u_0(x))\,dx,
\end{equation}
for $\Phi_\varepsilon(r)=\sigma^2_\varepsilon(r)r$ and $\Psi_\varepsilon(r)=\int_0^r\Phi_\varepsilon(\theta) \,d\theta$.
\end{lemma}

Now the existence of a weak solution to \eqref{eq:mode_sde_smooth} first  cou{l}d be classical{l}y reformulated into a martingale problem.
Owing to the boundedness of {$(t,x)\mapsto \sigma(u^\epsilon(t,x))$}, this gives (see Theorem 2.6, in Figalli \cite{Figalli-08}) the following result.
\begin{proposition}\label{prop:WellposednessSmoothsde}
Under \hypmodo, \hypmodi and \hypmodiiweaken,  there exists a unique weak solution $(X^\varepsilon_t;\,0\leq t\leq T)$ to \eqref{eq:mode_sde_smooth} such that the time marginal densities are given by $(u^\varepsilon(t);\,0\leq t\leq T)$ from Lemma \ref{lem:WellposednessSmoothpde}.
\end{proposition}

\subsubsection{Proof of Lemma \ref{lem:linear_pde}}
For any $v\in\LH$, the identity \eqref{eq:mode_nondegen_dif} ensures that{, for a.e. $t\in (0,T)$,} the bilinear mapping
$$
(u_1,u_2)\mapsto {\Ll_t}(u_1,u_2)={\int_{\er^d} {\frac{1}{2}\alpha_\varepsilon}(v(t,x))\left(\nabla_x u_1(x)\cdot \nabla_x u_2(x)\right)\,dx}
$$
{is continuous  on $H^1(\er^d)\times H^1(\er^d)$, since $\Ll_t(u_1,u_2)\leq \tfrac{1}{2}
(\sup_{0\leq r\leq \Vert v\Vert_{L^\infty}}{\alpha_\varepsilon}(r))\Vert\nabla_x u_1\Vert_{L^2(\er^d)}\Vert\nabla_xu_2\Vert_{L^2(\er^d)}$. Moreover, since $\Ll_t(u,u)\geq \frac{\epsilon}{2}\Vert\nabla_xu\Vert^2_{L^2(\er^d)}$, $\Ll_t$ satisfies the hypothesis of Theorem 1.1 in Lions~\cite[Chapter 4]{Lions-61} : $\Ll_t(u,u) + \frac{\epsilon}{2} \|u^2\|_{L^2(\er^d)} \geq    \frac{\epsilon}{2}  \|u^2\|_{H^1(\er^d)}$ for all $u\in H^1(\er^d)$.}  Applying Theorem 1.1 and Lemma 1.1 in \cite{Lions-61}, we deduce the existence of a
solution $\overline{u}^\varepsilon\in L^2((0,T);H^1(\er^d))$ to
\begin{align*}
&\int_{\er^d}\overline{u}^\varepsilon(T_0,x)f(T_0,x)\,dx-\int_{\er^d}u_0(x)f(0,x)\,dx\\
&\qquad=\int_{(0,T_0)\times\er^d}\overline{u}^\varepsilon(t,x) \frac{\partial f}{\partial t}(t,x) -{\frac{1}{2}\alpha_\varepsilon}(v(t,x))
\left(\nabla_x\overline{u}^\varepsilon(t,x)\cdot \nabla_x f(t,x)\right) \,dt\,dx,\\
&\qquad\qquad\qquad \forall\,0\leq T_0\leq T, \forall\,f\in\Cc^\infty_c([0,T_0]\times\er^d).
\end{align*}
The property of $u^\varepsilon\in \Cc([0,T];L^2(\er^d))$ can be proved in the same way as in \cite{Lions-61}, Theorem $2.1$, Chapter $4$.

The energy estimate \eqref{eq:linear_ineq_energy} is obtained by adapting some arguments of Lady\v{z}enskaja {\it et al.} \cite{LadSolUra-68}, p.~141--142.
For $0\leq T_0\leq T$, for $h>0$ such that $T_0+h\leq T$, and $\eta_h\in \Cc^\infty_c((-h,T_0)\times\er^d)$ such that $\eta_h(t,x)=0$ whenever $t\leq 0$ or $t\geq T_0-h$,
we define
$$
\phi(t,x)=\frac{1}{h}\int_{t-h}^t\eta_h(s,x)\,ds=\int_{-1}^0\eta_h(t+sh,x)\,ds.
$$
Plugging $\phi$ as a test function into the weak formulation \eqref{eq:mode_pde_linear} gives
\begin{align*}
u_0(x)f(0,x)\,dx\\
&0=\int_{(0,T_0)\times\er^d}\overline{u}^\varepsilon(t,x)\partial_t \phi(t,x) -{\frac{1}{2}\alpha_\varepsilon}(v(t,x))
(\nabla_x\overline{u}^\varepsilon(t,x)\cdot \nabla_x \phi(t,x) )\,dt\,dx.
\end{align*}
Since $\partial_t\phi(t,x)=\frac{1}{h}(\eta_h(t,x)-\eta_h(t-h,x))$,
\begin{align*}
\int_{(0,T_0)\times\er^d}\overline{u}^\varepsilon(t,x)\partial_t \phi(t,x)\,dt\,dx
&=\frac{1}{h}\left(\int_{(0,T_0)\times\er^d}\overline{u}^\varepsilon(t,x)\eta_h(t,x)\,dt\,dx-\int_{(0,T_0)\times\er^d}\overline{u}^\varepsilon(t,x)\eta_h(t-h,x)\,dt\,dx\right)\\
&=\frac{1}{h}\left(\int_{(0,T_0)\times\er^d}\overline{u}^\varepsilon(t,x)\eta_h(t,x)\,dt\,dx-\int_{(0,T_0)\times\er^d}\overline{u}^\varepsilon(t+h,x)\eta_h(t,x)\,dt\,dx\right)\\
&=\int_{(0,T_0)\times\er^d}\partial_t\overline{u}^\varepsilon_h(t,x)\eta_h(t,x)\,dt\,dx,
\end{align*}
for $\overline{u}^\varepsilon_h(t,x)=\frac{1}{h}\int_{t}^{t+h}\overline{u}^\varepsilon(s,x)\,ds$.
In the same manner, we have
\begin{align*}
\int_{(0,T_0)\times\er^d}{\frac{1}{2}\alpha_\varepsilon}(v(t,x))(\nabla_x\overline{u}^\varepsilon(t,x)\cdot \nabla_x \phi(t,x)) \,dt\,dx
=\int_{(0,T_0)\times\er^d}\frac{1}{2}\left({\alpha_\varepsilon}(v)\nabla_x\overline{u}^\varepsilon\right)_h(t,x)\cdot \nabla_x \eta_h(t,x) \,dt\,dx
\end{align*}
where
$$
\left({\alpha_\varepsilon}(v)\nabla_x\overline{u}^\varepsilon\right)_h(t,x):=
\frac{1}{h}\int_{t}^{t+h}{\alpha_\varepsilon}(v(s,x))\nabla_x\overline{u}^\varepsilon(s,x)\,ds.
$$
Therefore
\begin{align*}
0=\int_{(0,T_0)\times\er^d}\partial_t\overline{u}^\varepsilon_h(t,x) \eta_h(t,x) +\frac{1}{2}\left({\alpha_\varepsilon}(v)\nabla_x\overline{u}^\varepsilon\right)_h(t,x)\cdot \nabla_x \eta_h(t,x) \,dt\,dx.
\end{align*}
Extending the previous equality from $\eta_h\in\Cc^{\infty}_c((0,T_0-h)\times\er^d)$ to $\eta\in L^{2}((0,T_0);H^1(\er^d))$ by density, it follows that
\begin{equation}
\label{proofst:linear_mode_pde_1}
0=\int_{(0,T_0)\times\er^d}\partial_t\overline{u}^\varepsilon_h(t,x) \eta(t,x) +\frac{1}{2}\left({\alpha_\varepsilon}(v)\nabla_x\overline{u}^\varepsilon\right)_h(t,x)\cdot \nabla_x \eta(t,x) \,dt\,dx.
\end{equation}
Next replacing $\eta=\overline{u}^\varepsilon_h$, we get
\begin{align*}
0&=\int_{(0,T_0)\times\er^d}\partial_t\overline{u}^\varepsilon_h(t,x) \overline{u}^\varepsilon_h(t,x) +\frac{1}{2}\left({\alpha_\varepsilon}(v)\nabla_x\overline{u}^\varepsilon\right)_h(t,x)\cdot \nabla_x \overline{u}^\varepsilon_h(t,x) \,dt\,dx\\
&=\int_{\er^d}\left(\overline{u}^\varepsilon_h(T_0,x)\right)^2\,dx-\int_{\er^d}\left(\overline{u}^\varepsilon_h(0,x)\right)^2\,dx+
\int_{(0,T_0)\times\er^d} \frac{1}{2}\left({\alpha_\varepsilon}(v)\nabla_x\overline{u}^\varepsilon\right)_h(t,x)\cdot \nabla_x \overline{u}^\varepsilon_h(t,x) \,dt\,dx.
\end{align*}
Since
\begin{align*}
&\left|\int_{\er^d} \left(\overline{u}^\varepsilon_h(T_0,x)\right)^2\,dx-\int \left(\overline{u}^\varepsilon(T_0,x)\right)^2\,dx\right|\\
&=\left|\int_{\er^d} \left(\int_0^1\left(\overline{u}^\varepsilon(T_0+hs,x)\,ds-\overline{u}^\varepsilon(T_0,x)\right)\right)
\left(\int_0^1\left(\overline{u}^\varepsilon(T_0+hs,x)\,ds+\overline{u}^\varepsilon(T_0,x)\right)\right)\,dx\right|\\
&\qquad\leq \sqrt{\int_0^1\Vert \overline{u}^\varepsilon(T_0+hs)-\overline{u}^\varepsilon(T_0)\Vert^2_{L^2(\er^d)}\,ds}
\sqrt{\int_0^1\Vert \overline{u}^\varepsilon(T_0+hs)+\overline{u}^\varepsilon(T_0)\Vert^2_{L^2(\er^d)}\,ds},
\end{align*}
we have
\begin{equation}\label{proofst:linear_mode_pde_2}
\lim_{h\rightarrow 0}\int_{\er^d}\left(\overline{u}^\varepsilon_h(T_0,x)\right)^2\,dx=\int_{\er^d}\left(\overline{u}^\varepsilon(T_0,x)\right)^2\,dx.
\end{equation}
Similarly,
$$
\lim_{h\rightarrow 0}\int_{\er^d}\left(\overline{u}^\varepsilon_h(0,x)\right)^2\,dx=\int_{\er^d}\left(\overline{u}^\varepsilon_0(x)\right)^2\,dx,
$$
and
$$
\lim_{h\rightarrow 0}\int_{(0,T_0)\times\er^d} \left({\frac{1}{2}\alpha_\varepsilon}(v)\nabla_x\overline{u}^\varepsilon\right)_h(t,x)\cdot \nabla_x \overline{u}^\varepsilon_h(t,x) \,dt\,dx=\int_{(0,T_0)\times\er^d} {\frac{1}{2}\alpha_\varepsilon}(v(t,x))\left|\nabla_x \overline{u}^\varepsilon_h(t,x)\right|^2 \,dt\,dx,
$$
From which we deduce \eqref{eq:linear_ineq_energy}.

The non-negativeness of $\overline{u}^\varepsilon$ and \eqref{eq:linear_ineq_Linfini} follows from comparison principles:  since $\overline{u}^\varepsilon$ is in $L^2((0,T);H^1(\er^d))$, its negative part $(\overline{u}^\varepsilon)^-=\max(0,-\overline{u}^\varepsilon))$ is also in $L^2((0,T);H^1(\er^d))$ and
$\nabla_x(\overline{u}^\varepsilon)^-=-\nabla_x\overline{u}^\varepsilon\ind_{\{\overline{u}^\varepsilon\leq 0\}}$. Taking $\eta=(\overline{u}^\varepsilon)^-$  in  \eqref{proofst:linear_mode_pde_1} yields
\begin{align*}
0&=\int_{(0,T_0)\times\er^d}\partial_t\overline{u}^\varepsilon_h(t,x)(\overline{u}^\varepsilon)^-(t,x) +\frac{1}{2}\left({\alpha_\varepsilon}(v)\nabla_x\overline{u}^\varepsilon\right)_h(t,x)\cdot \nabla_x (\overline{u}^\varepsilon)^-(t,x) \,dt\,dx\\
&=-\int_{\er^d}\left(\left(\overline{u}^\varepsilon_h\right)^-(T_0,x)\right)^2\,dx+
\int_{\er^d}\left(\left(\overline{u}^\varepsilon_h\right)^-(0,x)\right)^2\,dx+
\int_{(0,T_0)\times\er^d}\frac{1}{2}\left({\alpha_\varepsilon}(v)\nabla_x\overline{u}^\varepsilon\right)_h(t,x)\cdot \nabla_x (\overline{u}^\varepsilon)^-(t,x) \,dt\,dx.
\end{align*}
Replicating the same arguments as for \eqref{proofst:linear_mode_pde_2},
\begin{align*}
\lim_{h\rightarrow 0}\int_{\er^d}\left(\left(\overline{u}^\varepsilon_h\right)^-(T_0,x)\right)^2\,dx&=\int_{\er^d}\left(\left(\overline{u}^\varepsilon\right)^-(T_0,x)\right)^2dx,\\
\lim_{h\rightarrow 0}\int_{\er^d}\left(\left(\overline{u}^\varepsilon_h\right)^-(0,x)\right)^2\,dx&=\int_{\er^d}\left(\left(u_0\right)^-(x)\right)^2dx,\\
\lim_{h\rightarrow 0}\int_{(0,T_0)\times\er^d}\frac{1}{2}\left({\alpha_\varepsilon}(v)\nabla_x\overline{u}^\varepsilon\right)_h(t,x)\cdot \nabla_x \left(\overline{u}^\varepsilon_h\right)^-(t,x) \,dt\,dx
&=\int_{(0,T_0)\times\er^d}\frac{1}{2}{\alpha_\varepsilon}(v(t,x))\nabla_x\overline{u}^\varepsilon(t,x)\cdot \nabla_x \left(\overline{u}^\varepsilon_h\right)^-(t,x) \,dt\,dx\\
&=-\int_{(0,T_0)\times\er^d}\frac{1}{2}{\alpha_\varepsilon}(v(t,x))\left|\nabla_x \left(\overline{u}^\varepsilon_h\right)^-(t,x)\right|^2 \,dt\,dx.
\end{align*}
Consequently,
\begin{align*}
0=-\int_{\er^d}\left(\left(\overline{u}^\varepsilon\right)^-(T_0,x)\right)^2\,dx+
\int_{\er^d}\left(\left(u_0\right)^-(x)\right)^2\,dx-
\int_{(0,T_0)\times\er^d}{\frac{1}{2}\alpha_\varepsilon}(v(t,x))\left|\nabla_x \left(\overline{u}^\varepsilon_h\right)^-(t,x)\right|^2 \,dt\,dx.
\end{align*}
Since $u_0\geq 0$ and ${\alpha_\varepsilon}$ is non-negative, we deduce immediately that $\left(\overline{u}^\varepsilon\right)^-=0$ on $(0,T)\times\er^d$.

For the proof of \eqref{eq:linear_ineq_Linfini}, we proceed similarly: we set $K=\Vert u_0\Vert_{L^\infty}$.  Then the positive part of $\overline{u}^\varepsilon(t,x)-K$,
$$
\left(\overline{u}^\varepsilon(t,x)-K\right)^+
$$
is in $L^2((0,T);H^1(\er^d))$. This statement simply follows from the  observation  that
$$
\left|\left(\overline{u}^\varepsilon(t,x)-K\right)^+\right|^2
\leq \left|\overline{u}^\varepsilon(t,x)\right|^2,\,\,\mbox{a.e. on}\,(0,T)\times\er^d,
$$
and that
$$
\nabla_x\left(\overline{u}^\varepsilon-K\right)^+=\nabla_x\overline{u}^\varepsilon\1_{\{\overline{u}^\varepsilon-K\geq 0\}}.
$$
By plugging $\eta=\left(\overline{u}^\varepsilon-K\right)^+$ into \eqref{proofst:linear_mode_pde_1}, we get
\begin{align*}
0&=\int_{(0,T_0)\times\er^d}\partial_t\overline{u}^\varepsilon_h(t,x) \left(\overline{u}^\varepsilon(t,x)-K\right)^+(t,x) +\frac{1}{2}\left({\alpha_\varepsilon}(v)\nabla_x\overline{u}^\varepsilon\right)_h(t,x)\cdot \nabla_x \left(\overline{u}^\varepsilon(t,x)-K\right)^+(t,x) \,dt\,dx.
\end{align*}
Taking the limit $h\rightarrow 0$ of the above expression  yields
\begin{align*}
0&=\int_{\er^d}\left(\left(\overline{u}^\varepsilon(t,x)-K\right)^+\right)^2(t,x)\,dx-\int_{\er^d}\left(\left(u_0(x)-K\right)^+\right)^2(t,x)\,dx\\
&\quad +\int_{(0,T_0)\times\er^d}{\frac{1}{2}\alpha_\varepsilon}(v(t,x))\left|\nabla_x \left(\overline{u}^\varepsilon(t,x)-K\right)^+\right|^2(t,x) \,dt\,dx.
\end{align*}
Since $\left(u_0(x)-K\right)^+=0$ for a.a. $x\in\er^d$, \eqref{eq:linear_ineq_Linfini} follows.
\subsubsection{Proof of Lemma \ref{lem:WellposednessSmoothpde}}

The existence of a weak solution to the non linear PDE  \eqref{eq:mode_pde_smooth} will be deduced from the following fixed point theorem that we apply to the mapping $A:v\in \Xx\mapsto A(v)\in \Xx$, for
\[
\Xx=\left\{v\in \LH\,:\,\,v\geq 0\,\text{ a.e. on}\,(0,T)\times\er^d\,\text{and}\,\Vert v\Vert_{L^\infty((0,T)\times\er^d)}\leq \Vert u_0\Vert_{L^\infty(\er^d)}\right\}
\]
equipped with the $\Vert\,\Vert_{\LH}$-norm, and where $A$ assigns to any nonnegative $v\in \Xx$, the weak solution
 $A(v)$ to the linear PDE \eqref{eq:mode_pde_linear} given by Lemma \ref{lem:linear_pde} with the estimates \eqref{eq:linear_ineq_Linfini} and \eqref{eq:linear_ineq_energy}.

\begin{theorem}[Schaefer's fixed point Theorem, see \cite{Evans-97}, Theorem 4, Chapter $9$, Section $2$]
 Let $\Xx$ be a Banach space and $A:\Xx\rightarrow \Xx$ be a continuous and compact mapping such that the set
\begin{align*}
 {\Ii=}\left\{u\in \Xx\mbox{ s.t. there exists}\,0\leq \lambda\leq 1\,\text{with }\,u=\lambda A(u)\right\}
\end{align*}
 is bounded in $\Xx$. Then $A$ has a fixed point in $\Xx$.
\end{theorem}

\paragraph{The continuity of $A$. } Let $\{v_n\}_{n}$ converg{e} to $v_\infty$ in $L^2((0,T);H^1(\er^d))$. Then, since $A(v_n)$ and $A(v_\infty)$ are weak solution to \eqref{eq:mode_pde_linear} {endowed with} the diffusion coefficient $\sigma_\varepsilon(v_n)$ and $\sigma_\varepsilon(v_\infty)$ respectively, we have
\[
0=\int_{(0,T)\times\er^d}\left(A(v_n)-A(v_\infty)\right)\partial_t f-\left({\frac{1}{2}\alpha_\varepsilon}(v_n)\nabla_xA(v_n)-{\frac{1}{2}\alpha_\varepsilon}(v_\infty)\nabla_xA(v_\infty)\right)\nabla_x f,
\]
 for all $f\in\Cc^\infty_c((0,T)\times\er^d)$.
 Replicating the arguments for the energy estimate \eqref{eq:linear_ineq_energy} in Lemma \ref{lem:linear_pde}, we deduce that
 \begin{align*}
 &\int_{\er^d} \left|A(v_n)(T_0)-A(v_\infty)(T_0) \right|^2+\int_{(0,T_0)\times\er^d} {\frac{1}{2}\alpha_\varepsilon}(v_n)\left|\nabla_xA(v_n)-\nabla_xA(v_\infty) \right|^2\\
 &=\int_{(0,T_0)\times\er^d}\frac{1}{2}\left({\alpha_\varepsilon}(v_n)-{\alpha_\varepsilon}(v_\infty)\right)\left(\nabla_xA(v_n)-\nabla_xA(v_\infty) \right)\cdot \nabla_xA(v_n)\\
 &\qquad\leq\frac{1}{\varepsilon} \int_{(0,T_0)\times\er^d}\frac{1}{2}\left|{\alpha_\varepsilon}(v_n)-{\alpha_\varepsilon}(v_\infty)\right|^2 \left|\nabla_xA(v_n)\right|^2
 +\frac{\varepsilon}{4}\int_{(0,T_0)\times\er^d}\left|\nabla_xA(v_n)-\nabla_xA(v_\infty) \right|^2.
 \end{align*}
 Since $v_n\rightarrow v_\infty$ in $L^2((0,T)\times\er^d)$, there exists a subsequence $\{v_{n_k}\}_k$ such that $v_{n_k}\rightarrow v_\infty$ a.e. on $(0,T)\times\er^d$ and such that $\sup_k|v_{n_k}|$ is in $L^2((0,T)\times\er^d)$. Replacing $v_n$ by $v_{n_k}$ in the preceding inequality and applying the Lebesgue's dominated convergence theorem, we deduce that $\lim_k A(v_{n_k})=A(v_\infty)$ in $\LH$.
  With the same reasoning, for any subsequence of $\{A(v_{n})\}_n$, we can extract a subsequence which converges to $A(v_\infty)$. {Since $\Xx$ is a closed subset of $\LH$, t}his implies the continuity of $A$.

\paragraph{The compactness of $A$. } Owing to Lemma \ref{lem:linear_pde}, for any converging sequence $\{v_n\}_n$ in {$\Xx$, }
$\{A(v_n)\}_n$ is a bounded sequence in $\LH$. We can then extract a subsequence $\{A(v_{n_k})\}_k$  converging to some limit $A_\infty$ {in} the weak topology of $\LH$. 
In particular, since $\sup_k\Vert A(v_{n_k})\Vert_{L^2((0,T)\times\er^d)}$ is finite, $\Vert {\nabla_x} A(v_{n_k})- A_\infty\Vert_{L^2((0,T)\times\er^d)}$ tends to $0$ as $k$ tends to $\infty$. {The uniform bound \eqref{eq:linear_ineq_Linfini} naturally holds for $v_\infty$.}

\paragraph{The  boundedness of {$\Ii$}. }  Finally, let us consider the set
   $$
  \Ii:=\left\{u\in {\Xx}\,\mbox{ s.t. there exists}\,0\leq \lambda\leq 1\,\text{ with }\,u=\lambda A(u)\right\}.
  $$
  Excluding the trivial case $\lambda=0$, one can check that for all $u_\lambda\in {\Xx}$ such that $u_\lambda=\lambda A(u_\lambda)$, $u_\lambda$ is {the} $\LH$-weak solution to
  \begin{equation*}
\left\{
\begin{aligned}
&\frac{1}{\lambda}\frac{\partial u_\lambda}{\partial t } - \frac{1}{2}\triangle (\sigma_\varepsilon^2(u_\lambda) \frac{u_\lambda}{\lambda}) = 0,{\text{ on }(0,T)\times\er^d},
  \\
&\frac{1}{\lambda}u_\lambda({0,x})=u_0{(x)},~{x\in\er^d},
\end{aligned}
\right.
\end{equation*}
given as in Lemma \ref{lem:linear_pde}. Equivalently,
  \begin{align*}
&\int_{\er^d}u_\lambda(T_0,x)f(T_0,x)\,dx-\lambda\int_{\er^d}u_0(x)f(0,x)\,dx\\
&\qquad=\int_{(0,T_0)\times\er^d} u_\lambda(t,x) \partial_tf(t,x) -{\frac{1}{2}\alpha_\varepsilon}(u_\lambda)\nabla_x u_\lambda(t,x)\cdot \nabla_x f(t,x) \,dt\,dx.
\end{align*}
  The energy estimate \eqref{eq:linear_ineq_energy} from Lemma \ref{lem:linear_pde} then ensures that $\Vert u_\lambda\Vert^2_{L^2((0,T)\times\er^d)}\leq \lambda^2 T\Vert u_0\Vert^2_{L^2(\er^d)}$ from which we conclude on the boundedness of {$\Ii$}.

\medskip
The Schaefer Theorem ensures the existence of a $\LH$-weak solution to \eqref{eq:mode_pde_smooth}, for which the $L^2$ (and $L^\infty$) estimate established in Lemma \ref{lem:linear_pde} still hold true.

For the estimate \eqref{eq:Vasqueztype_ineq}, we replicate the proof arguments of Vasquez \cite{Vasquez-06} (see Chapter 5).  Owing to assumptions \hypmodo
   and \hypmodi,
   \begin{align*}
   \int_{\er^d} \Psi_\varepsilon(u_0(x))\,dx&=\int_{\er^d} \int_0^{u_0(x)}\Phi_\varepsilon(r)\,dr\,dx\leq \frac{1}{2}\left({\sup_{0\leq r\leq \Vert u_0\Vert_{L^\infty(\er^d)}}(\sigma(r))^2}+\varepsilon\right)
   \Vert u_0\Vert^2_{L^2(\er^d)}.
   \end{align*}
   Replicating the proof arguments of \eqref{proofst:linear_mode_pde_1}, for all $\eta\in\Cc^\infty_c((0,T)\times\er^d)$, for $h>0$, $T_0>0$ such that $T_0+h\leq T$,
   \begin{equation}\label{proofst:nonlinear_mode_pde_1}
0=\int_{(0,T_0)\times\er^d}\partial_tu^\varepsilon_h(t,x) \eta(t,x) +\frac{1}{2}\left({\alpha_\varepsilon}(u^\varepsilon)\nabla_x u^\varepsilon\right)_h(t,x)\cdot \nabla_x \eta(t,x) \,dt\,dx,
\end{equation}
with  $u^\varepsilon_h(t,x):=\frac{1}{h}\int_{t}^{t+h}u^\varepsilon(s,x)\,ds$, and
$$
\left({\alpha_\varepsilon}(u^\varepsilon)\nabla_x u^\varepsilon\right)_h(t,x):=\frac{1}{h}\int_{t}^{t+h}{\alpha_\varepsilon}(u^\varepsilon(s,x))\nabla_x u^\varepsilon(s,x)\,ds.
$$
Observing that
$$
\Vert \Phi_\varepsilon(u^\varepsilon_h)\Vert_{L^2((0,T)\times\er^d)}=\Vert \left(\sigma^2(u^\varepsilon)+\varepsilon\right)u^\varepsilon_h\Vert_{L^2((0,T)\times\er^d)}\leq
\frac{1}{2}\left( {\sup_{0\leq r\leq \Vert u_0\Vert_{L^\infty(\er^d)}}(\sigma(r))^2}+\varepsilon\right)
   \Vert u^\varepsilon\Vert_{L^2((0,T)\times\er^d)},
$$
and that
\begin{align*}
\Vert \nabla_x\Phi_\varepsilon(u^\varepsilon_h)\Vert^{{2}}_{L^{2}((0,T)\times\er^d)}&=\Vert \Phi'_\varepsilon(u^\varepsilon_h) \nabla_xu^\varepsilon_h\Vert^{{2}}_{L^{2}((0,T)\times\er^d)}=\Vert {\alpha_\varepsilon}(u^\varepsilon)\nabla_xu^\varepsilon_h\Vert^{{2}}_{L^{2}((0,T)\times\er^d)}\\
&\leq\left(\Vert {\alpha(u^\epsilon)}\Vert^2_{L^\infty(\er^d)}+\varepsilon^{{2}}\right)
   \Vert \nabla_x u^\varepsilon\Vert^2_{L^2((0,T)\times\er^d)},
\end{align*}
$\Phi_\varepsilon(u^\varepsilon_h)$ is in $\LH$. Taking $\eta=\Phi_\varepsilon(u^\varepsilon_h)$ in \eqref{proofst:nonlinear_mode_pde_1}
\begin{align*}
0&=\int_{(0,T_0)\times\er^d}\partial_tu^\varepsilon_h(t,x) \Phi_\varepsilon(u^\varepsilon_h)(t,x) +\frac{1}{2}\left({\alpha_\varepsilon}(u^\varepsilon)\nabla_x u^\varepsilon\right)_h(t,x)\cdot \nabla_x \Phi_\varepsilon(u^\varepsilon_h)(t,x) \,dt\,dx\\
&= \int_{(0,T_0)\times\er^d}\partial_t\Psi_\varepsilon(u^\varepsilon_h)(t,x)+\frac{1}{2}\int_{(0,T_0)\times\er^d}\left(\nabla_x \Phi_\varepsilon(u^\varepsilon_h)\right)_h(t,x)\cdot \nabla_x \Phi_\varepsilon(u^\varepsilon_h)(t,x) \,dt\,dx\\
&=\Vert \Psi_\varepsilon(u^\varepsilon_h)(T_0)\Vert_{L^1(\er^d)}-\Vert \Psi_\varepsilon(u^\varepsilon_h)(0)\Vert_{L^1(\er^d)}+\frac{1}{2}\int_{(0,T_0)\times\er^d}\left(\nabla_x \Phi_\varepsilon(u^\varepsilon_h)\right)_h(t,x)\cdot \nabla_x \Phi_\varepsilon(u^\varepsilon_h)(t,x) \,dt\,dx.
\end{align*}
Taking the limit $h\rightarrow 0$ it follows that
\begin{align*}
0&=\Vert \Psi_\varepsilon(u^\varepsilon)(T_0)\Vert_{L^{{1}}(\er^d)}-\Vert \Psi_\varepsilon(u_0)\Vert_{L^{{1}}(\er^d)}+\frac{1}{2}\int_{(0,T_0)\times\er^d}\left|\nabla_x \Phi_\varepsilon(u^\varepsilon_h)\right|^2(t,x)\,dt\,dx.
\end{align*}
{Finally, since $u^\epsilon$ is a $L^2((0,T);H^1(\er^d))$-solution to \eqref{eq:mode_pde_smooth}, the property $u^\epsilon\in\Cc([0,T];L^2(\er^d))$ can be again obtained by following \cite{Lions-61}, Theorem 2.1, Chapter $4$.} This ends the proof.

\subsection{Existence result for \eqref{eq:moderated}}

The existence result in Theorem \ref{thm:moderatedexist} will be deduced from the asymptotic behavior  (up to a subsequence extraction) of the solution to \eqref{eq:mode_sde_smooth} as $\varepsilon\rightarrow 0$. The precise result is the following:
\begin{proposition}\label{prop:limit_nonlinear_smooth_sdepde} Assume  \hypmodo, \hypmodi and \hypmodii.  Consider $(X^\varepsilon_t;\,t\geq 0)$ solution to \eqref{eq:mode_sde_smooth} given by Proposition \ref{prop:WellposednessSmoothsde}. The sequence $\{(P^\varepsilon,u^\varepsilon)\}_{\varepsilon>0}$, defined by
$$
P^\varepsilon=\law(X^\varepsilon_t;\,0\leq t\leq T),
$$
and $u^\varepsilon$ given as in Lemma \ref{lem:WellposednessSmoothpde},  admits a weakly converging subsequence $\{(P^{\varepsilon_k},u^{\varepsilon_k})\}_k$ such that $u^0 = \lim_ku^{\varepsilon_k}$ in $L^2((0,T)\times\er^d)$,  is a $L^2((0,T)\times\er^d)$-weak solution to
\begin{align}\label{eq:mode_pde_nonlinear}
\left\{
\begin{aligned}
&\frac{\partial u}{\partial t } - \frac{1}{2}\triangle_x (\sigma^2(u) u) = 0,{\text{ on }(0,T)\times\er^d,}\\
&u({0,x})=u_0{(x)},~{x\in\er^d,}
\end{aligned}
\right.
\end{align}
 and $P^{0}=\lim_k P^{\varepsilon_k}$ is solution to the following martingale problem \textit{(MP)}: let $(x(t);\,0\leq t\leq T)$  denotes
 the canonical process on $\Cc([0,T];\er^d)$, we have
\begin{description}
\item[\textit{(MP)-(i)}] For all $0\leq t\leq T$, $P^0(x(t)\in\,dx)=u^0(t,x)\,dx$;

\item[\textit{(MP)-(ii)}] For all $f\in\Cc^2_c(\er^d)$,
$$
t\mapsto f(x(t))-f(x(0))-\frac{1}{2}\int_0^t \sigma^2(u^0(s,x))\triangle_x f(x(s))\,ds
$$
is a continuous martingale under $P^0$.
\end{description}
\end{proposition}

From Lemma \ref{lem:WellposednessSmoothpde}, the $L^2((0,T)\times\er^d)$-convergence of   $u^{\varepsilon}$ to $u^0$ ensures that
\[
\Vert u^0\Vert_{L^\infty((0,T)\times\er^d)}<\infty,\,\mbox{ and }\,\int_{\er^d} \left|\nabla\Phi_\varepsilon(u^0(t,x))\right|^2\,dt\,dx<\infty.
\]
In particular under \hypmodii, the control of $\Vert \nabla\Phi_\varepsilon(u^0(t,x))\Vert_{L^2((0,T)\times\er^d)}$ yields to the estimate:
\[
\eta\Vert\nabla_x u^0\Vert_{L^2((0,T)\times\er^d)}\leq \Vert\alpha_{\varepsilon}(u^0)\nabla_x u^0\Vert_{L^2((0,T)\times\er^d)}\leq \Vert\nabla_x \Phi_\varepsilon(u^0)\Vert_{L^2((0,T)\times\er^d)}<\infty.
\]
Therefore
\begin{corollary}\label{coro:limit_nonlinear_smooth_sdepde} The time marginal densities $u^0(t)$ of $(X^0_t;\,0\leq t\leq T)$ given in Proposition \ref{prop:limit_nonlinear_smooth_sdepde} are in $L^\infty((0,T)\times\er^d)\cap L^2((0,T);H^1(\er^d))$.
\end{corollary}
\begin{proof}[Proof of Proposition \ref{prop:limit_nonlinear_smooth_sdepde}]
Owing to \eqref{eq:smoothednonlinear_ineq_energy}, we deduce that $\{u^\varepsilon\}_{\varepsilon}$ is relatively compact for the weak topology in $L^2((0,T)\times\er^d)$. Denote by $\{u^{\varepsilon_k}\}_k$ a (weakly) converging subsequence and $u^0$ its limit.
Under the assumption \hypmodii, the estimate \eqref{eq:Vasqueztype_ineq} ensures that
\[
\sup_\varepsilon\Vert \nabla_xu^{\varepsilon}\Vert_{L^{2}((0,T)\times\er^d)}\leq \frac{1}{\eta}\sup_\varepsilon\Vert \alpha(u^\varepsilon)\nabla_xu^{\varepsilon}\Vert_{L^{2}((0,T)\times\er^d)}=\frac{1}{\eta}\sup_\varepsilon\Vert \nabla_x \Phi(u^{\varepsilon})\Vert_{L^{2}((0,T)\times\er^d)}<\infty,
\]
so that $\lim_k \Vert u^{\varepsilon_k}-u^0\Vert_{L^2((0,T)\times\er^d)}=0$. We can further extract a converging subsequence such that
the convergence holds a.e. on $(0,T)\times\er^d$ and $\sup_k|u^{\varepsilon_k}|\in L^2((0,T)\times\er^d)$. Since $\sigma_\varepsilon(u^\varepsilon)$ is bounded and by \hypmodo, according to the Kolmogorov-Centov criterion, the sequence $\{P^\varepsilon\}_\varepsilon$ is tight on $(\Cc([0,T];\er^d),\Bb(\Cc([0,T];\er^d)))$. Denot{e} for simplicity by $\{(P^{\varepsilon_k},u^{\varepsilon_k})\}_k$ a converging pair of $\{(P^{\varepsilon},u^{\varepsilon})\}_\varepsilon$ and by $(P^0,u^0)$ its limit. Since $P^{\varepsilon_k}(x(t)\in\,dx)=u^{\varepsilon_k}(t,x)\,dx$ the $L^2$-convergence {of} $u^{\varepsilon_k}$ {to} $u^0$ and the convergence of the time marginal distributions of $P^{\varepsilon_k}$ to $P^{0}$ ensure that $P^0(x(t)\in\,dx)=u^0(t,x)\,dx$ for a.e. $0\leq t\leq T$.

Coming back to \eqref{eq:mode_pde_smooth} and taking the limit $k\rightarrow 0$ in the expression,
\begin{align*}
0=\int_{(0,T)\times\er^d}u^\varepsilon(t,x)\partial_t f(t,x)-{\frac{1}{2}\alpha_\varepsilon}(u^\varepsilon(t,x)) \nabla_x u^\varepsilon(t,x)\cdot \nabla_x f(t,x) \,dt\,dx,
\end{align*}
for $f\in\Cc^\infty_c((0,T)\times\er^d)$, and owing to the continuity of $\sigma$, we deduce that $u^0$ is a weak $L^2((0,T)\times\er^d)$-solution to \eqref{eq:mode_pde_nonlinear}.

In order to identify $P^0$ as the solution of the martingale problem \textit{(MP)}, it is sufficient to show that, for all $0\leq s<t\leq T$, $\psi:\Cc([0,s];\er^d)\rightarrow \er$ bounded and continuous, $f\in\Cc^2_c(\er^d)$,
\begin{equation}
\label{proofst:nonlinearlimit_mode_pde_1}
\begin{aligned}
&\lim_k\EE_{P^{\varepsilon_k}}\left[\psi(x(r);0\leq r\leq s)\int_s^t\sigma^2_\varepsilon(u^\varepsilon(\theta,x(\theta)))\triangle_xf(x(\theta))\,d\theta\right]\\
&=\EE_{P^{0}}\left[\psi(x(r);0\leq r\leq s)\int_s^t\sigma^2(u^0(\theta,x(\theta)))\triangle_xf(x(\theta))\,d\theta\right].
\end{aligned}
\end{equation}
To this {end}, let us introduce a smooth approximation of $\sigma^2_\varepsilon(u^{\varepsilon_k})$ and $\sigma^2(u^{0})$ with
$$
\left(\sigma^2_\varepsilon(u^{\varepsilon_k}(t))\right)_\beta(x):=\left(\phi_\beta*\sigma^2_\varepsilon(u^{\varepsilon_k}(t))\right)(x),\text{ and }
\left(\sigma^2(u^{0}(t))\right)_\beta(x):=\left(\phi_\beta*\sigma^2(u^{0}(t))\right)(x),
$$
for $*$ denoting the convolution product on the variable $x\in\er^d$ and $\{\phi_\beta\}_{\beta>0}$ a sequence of mollifiers on $\er^d$ given by $\phi_\beta(y)=\frac{1}{\beta^d}\phi(\frac{y}{\beta})$ with $\phi\geq 0$, $\phi\in\Cc^\infty_c(\er^d)$ and $\int\phi(y)\,dy=1$.

Then, we can consider
\begin{align*}
&\left|\EE_{P^{\varepsilon_k}}\left[\psi(x(r);0\leq r\leq s)\int_s^t\sigma^2_\varepsilon(u^\varepsilon(\theta,x(\theta)))\triangle_xf(x(\theta))\,d\theta\right]\right.\\
&\quad\quad\quad\left.-\EE_{P^{0}}\left[\psi(x(r);0\leq r\leq s)\int_s^t\sigma^2(u^0(\theta,x(\theta)))\triangle_xf(x(\theta))\,d\theta\right]\right|\\
&\leq \left|\EE_{P^{\varepsilon_k}}\left[\psi(x(r);0\leq r\leq s)\int_s^t\left(\sigma^2_\varepsilon(u^\varepsilon(\theta,x(\theta)))-\left(\sigma^2_\varepsilon(u^{\varepsilon_k}(\theta))\right)_\beta(x(\theta))\right)
\triangle_xf(x(\theta))\,d\theta\right]\right|\\
&\quad +\left|\EE_{P^{\varepsilon_k}}\left[\psi(x(r);0\leq r\leq s)\int_s^t\left(\sigma^2_\varepsilon(u^{\varepsilon_k}(\theta))\right)_\beta(x(\theta))
\triangle_xf(x(\theta))\,d\theta\right]\right.\\
&\quad\quad\quad\left.-\EE_{P^{0}}\left[\psi(x(r);0\leq r\leq s)\int_s^t\left(\sigma^2(u^{0}(\theta))\right)_\beta(x(\theta))
\triangle_xf(x(\theta))\,d\theta\right]\right|\\
&\quad +\left|\EE_{P^{0}}\left[\psi(x(r);0\leq r\leq s)\int_s^t\left(\left(\sigma^2(u^{0}(\theta))\right)_\beta(x(\theta))
-\sigma^2(u^0(\theta,x(\theta)))\triangle_xf(x(\theta))\right)\,d\theta\right]\right|\\
&\quad\quad =:I^{\varepsilon,\beta}_1+I^{\varepsilon,\beta}_2+I^{\beta}_3.
\end{align*}
By the weak convergence of $P^{\varepsilon_k}$ and since $\left(\sigma^2_\varepsilon(u^{\varepsilon_k}(t))\right)_\beta(x)$ converges locally to $\left(\sigma^2(u^{0}(t))\right)_\beta(x)$, it follows that $\lim_{\beta\rightarrow 0}\lim_{\varepsilon\rightarrow 0} I^{\varepsilon,\beta}_2=0$.

For $I^{\varepsilon,\beta}_1$, observe that
\begin{align*}
I^{\varepsilon,\beta}_1&\leq \Vert \psi\Vert_{\infty}\EE_{P^{\varepsilon_k}}\left[\int_0^T\left|
\sigma^2_\varepsilon(u^\varepsilon(\theta,x(\theta)))-\left(\sigma^2_\varepsilon(u^{\varepsilon_k}(\theta))\right)_\beta(x(\theta))
\right|\left|\triangle_xf(x(\theta))\right|\,d\theta\right]\\
&\leq \Vert \psi\Vert_{\infty}\sup_k\Vert u^{\varepsilon_k}\Vert_{L^2((0,T)\times\er^d)}
\sqrt{\int_{(0,T)\times\er^d} \left|
\sigma^2_\varepsilon(u^\varepsilon(\theta,x))-\left(\sigma^2_\varepsilon(u^{\varepsilon_k}(\theta))\right)_\beta(x)
\right|^2\left|\triangle_xf(x)\right|^2\,dx\,d\theta}.
\end{align*}
Since
\begin{align*}
\left|\sigma^2_{\varepsilon_k}(u^{\varepsilon_k}(\theta,x))-\left(\sigma^2_{\varepsilon_k}(u^{{\varepsilon_k}}(\theta))\right)_\beta(x)\right|\leq \int_{\er^d}\phi(y)\left|\sigma^2_{\varepsilon_k}(u^{\varepsilon_k}(\theta,x))-\sigma^2_{\varepsilon_k}(u^{{\varepsilon_k}}(\theta,x-\beta y))\right|\,dy,
\end{align*}
we have
\begin{align*}
&\int_{(0,T)\times\er^d} \left|
\sigma^2_\varepsilon(u^\varepsilon(\theta,x))-\left(\sigma^2_\varepsilon(u^{\varepsilon_k}(\theta))\right)_\beta(x)
\right|^2\left|\triangle_xf(x)\right|^2\,dx\,d\theta\\
&\leq \int_{\er^d}\phi(y)\left(\int_{(0,T)\times\er^d} \left|
\sigma^2_{\varepsilon_k}(u^{\varepsilon_k}(\theta,x))- \sigma^2_{\varepsilon_k}(u^{{\varepsilon_k}}(\theta,x-\beta y))\right|^2\left|\triangle_xf(x)\right|^2\,d\theta\,dx
\right)\,dy.
\end{align*}
Then we observe that, for all $y\in\er^d$,
\begin{equation}\label{proofst:nonlinearlimit_mode_pde_2}
\begin{aligned}
&\int_{(0,T)\times\er^d} \left|
\sigma^2_{\varepsilon_k}(u^{\varepsilon_k}(\theta,x))- \sigma^2_{\varepsilon_k}(u^{{\varepsilon_k}}(\theta,x-\beta y))\right|^2\left|\triangle_xf(x)\right|^2\,d\theta\,dx\\
&\leq \int_{(0,T)\times\er^d} \left|
\sigma^2_{\varepsilon_k}(u^{\varepsilon_k}(\theta,x))- \sigma^2_{\varepsilon_k}(u^{0}(\theta,x))\right|^2\left|\triangle_xf(x)\right|^2\,d\theta\,dx\\
&\quad +\int_{(0,T)\times\er^d} \left|
\sigma^2_{\varepsilon_k}(u^0(\theta,x))- \sigma^2_{\varepsilon_k}(u^{0}(\theta,x-\beta y))\right|^2\left|\triangle_xf(x)\right|^2\,d\theta\,dx\\
&\quad +\int_{(0,T)\times\er^d} \left|
\sigma^2_{\varepsilon_k}(u^{\varepsilon_k}(\theta,x-\beta y))- \sigma^2_{\varepsilon_k}(u^{{\varepsilon_k}}(\theta,x-\beta y))\right|^2\left|\triangle_xf(x)\right|^2\,d\theta\,dx.
\end{aligned}
\end{equation}
By continuity of $\sigma$, as $k$ tends to $\infty$, $\sigma(u^{\varepsilon_k})$ tends to $\sigma(u^{0})$ a.e. on $(0,T)\times\er^d$. Therefore, by Lebesgue's dominated convergence theorem, the first expression in the right hand side of \eqref{proofst:nonlinearlimit_mode_pde_2} tends to $0$. In the same way
\begin{align*}
&\lim_k\int_{(0,T)\times\er^d} \left|
\sigma^2_{\varepsilon_k}(u^{\varepsilon_k}(\theta,x-\beta y))- \sigma^2_{\varepsilon_k}(u^{{\varepsilon_k}}(\theta,x-\beta y))\right|^2\left|\triangle_xf(x)\right|^2\,d\theta\,dx\\
&=\lim_k\int_{(0,T)\times\er^d} \left|
\sigma^2_{\varepsilon_k}(u^{\varepsilon_k}(\theta,x))- \sigma^2_{\varepsilon_k}(u^{{\varepsilon_k}}(\theta,x))\right|^2\left|\triangle_xf(x+\beta y)\right|^2\,d\theta\,dx=0.
\end{align*}
For the remaining component in \eqref{proofst:nonlinearlimit_mode_pde_2}, assuming that the support of $f$ is included in the open ball $B(0,R)$ for some radius $R<\infty$,
\begin{align*}
&\int_{(0,T)\times\er^d} \left|
\sigma^2_{\varepsilon_k}(u^0(\theta,x))- \sigma^2_{\varepsilon_k}(u^{0}(\theta,x-\beta y))\right|^2\left|\triangle_xf(x)\right|^2\,d\theta\,dx\\
&\leq \Vert\triangle_x f\Vert_{\infty}\int_{(0,T)\times B(0,R)} \left|
\sigma^2_{\varepsilon_k}(u^0(\theta,x))- \sigma^2_{\varepsilon_k}(u^{0}(\theta,x-\beta y))\right|^2\,d\theta\,dx
\end{align*}
Since $\sigma^2_{\varepsilon_k}(u^0)$ is bounded, the continuity of
$$
z\in\er^d\mapsto \int_{(0,T)\times B(0,R)} \left|
\sigma^2_{\varepsilon_k}(u^0(\theta,x))- \sigma^2_{\varepsilon_k}(u^{0}(\theta,x-z))\right|^2\,d\theta\,dx
$$
ensures that
$$
\lim_{\beta\rightarrow 0}\int_{(0,T)\times\er^d} \left|
\sigma^2_{\varepsilon_k}(u^0(\theta,x))- \sigma^2_{\varepsilon_k}(u^{0}(\theta,x-\beta y))\right|^2\left|\triangle_xf(x)\right|^2\,d\theta\,dx=0.
$$
Coming back to \eqref{proofst:nonlinearlimit_mode_pde_2}, we deduce that
$$
\lim_{\beta\rightarrow 0}\lim_{\varepsilon\rightarrow 0}\int_{(0,T)\times\er^d} \left|
\sigma^2_{\varepsilon_k}(u^{\varepsilon_k}(\theta,x))- \sigma^2_{\varepsilon_k}(u^{{\varepsilon_k}}(\theta,x-\beta y))\right|^2\left|\triangle_xf(x)\right|^2\,d\theta\,dx=0,
$$
and by extension that $\lim_{\beta\rightarrow 0}\lim_{\varepsilon\rightarrow 0} I^{\varepsilon,\beta}_1=0$.

Finally, for $I^\beta_3$, replicating the arguments for $I^{\varepsilon_k,\beta}_1$, we have
\begin{align*}
&\left|\EE_{P^{0}}\left[\psi(x(r);0\leq r\leq s)\int_s^t\left(\left(\sigma^2(u^{0}(\theta))\right)_\beta(x(\theta))
-\sigma^2(u^0(\theta,x(\theta)))\triangle_xf(x(\theta))\right)\,d\theta\right]\right|\\
&\leq \Vert \psi\Vert_{\infty}\sup_k\Vert u^{\varepsilon_k}\Vert_{L^2((0,T)\times\er^d)}\sqrt{\int_{(0,T)\times\er^d} \left|
\sigma^2(u^0(\theta,x))-\left(\sigma^2(u^{0}(\theta))\right)_\beta(x)
\right|^2\left|\triangle_xf(x)\right|^2\,dx\,d\theta}\\
&\leq \Vert \psi\Vert_{\infty}\sup_k\Vert u^{\varepsilon_k}\Vert_{L^2((0,T)\times\er^d)}\sqrt{\int_{\er^d}\phi(y)\int_{(0,T)\times\er^d} \left|
\sigma^2(u^0(\theta,x))-\sigma^2(u^{0}(\theta))(x-\beta y)
\right|^2\left|\triangle_xf(x)\right|^2\,dx\,d\theta\,dy}
\end{align*}
where the last upper bound tends to $0$ as $k\rightarrow \infty$.
We then conclude on \eqref{proofst:nonlinearlimit_mode_pde_1}.
\end{proof}

\subsection{Uniqueness result for \eqref{eq:moderated}}
Let us first start by showing that the time marginal distribution of \eqref{eq:moderated} are unique.
\begin{proposition}\label{prop:moderateduniq} Assume that \hypmodo, \hypmodi and \hypmodii hold true. Let $(X^1_t,\,u^1_t;\,t\in[0,T])$ and $(X^2_t,\,u^2_t;\,t\in[0,T])$ be two weak solutions to \eqref{eq:moderated}. Then, for all $0\leq t\leq T$, $u^1_t=u^2_t$ a.e. on $\er^d$.
This conclusion holds true also under \hypmodo, \hypmodi and \hypmodiiweaken plus the assumption that $\alpha$ is strictly increasing.
\end{proposition}
\begin{proof}
We give the proof assuming \hypmodo, \hypmodi and \hypmodiiweaken plus the assumption that $\alpha$ is strictly increasing, the other case can be easily to deduce from the following arguments.
Consider two weak solutions $(X^1_t,\,u^1_t;\,t\in[0,T])$ and $(X^2_t,\,u^2_t;\,t\in[0,T])$ to \eqref{eq:moderated}. Given some $\gamma>0$, that will be chosen later, define
\begin{equation}\label{eq:Gaussiankernel}
 G^\gamma_{s,t}(x,y)=G^\gamma_{t-s}(x-y)=\left(2\pi\gamma^2(t-s)\right)^{-\frac{d}{2}}\exp\left(-\frac{|x-y|^2}{2\gamma^2(t-s)}\right).
 \end{equation}
 Since the kernel $G^\gamma_{t}$ gives the fundamental solution related to the parabolic operator $\partial_s+\frac{\gamma^2}{2}\triangle_x$; that is
 \[
 \partial_s G^\gamma_{s,t}+\frac{\gamma^2}{2}\triangle_xG^\gamma_{s,t}=0,\,\lim_{s\rightarrow t^-}G^\gamma_{s,t}=\delta_{\{y\}},\,\delta_{\{y\}}\,\text{the Dirac measure in }y,
 \]
the function $v(s,x)={G^\gamma_{s,t}(f)(x)=}G^\gamma_{t-s}{*} f(x),\,s\leq t,x\in\er^d,f\in\Cc^\infty_c(\er^d)$ (${*}$ denoting the convolution product on $\er^d$) is a $\Cc^\infty([0,t]\times\er^d)$ function satisfying:
\begin{equation*}
\left\{
\begin{aligned}
&\partial_s v(s,x) + \frac{\gamma^2}{2}\triangle_x v(s,x) = 0, ~0\leq s< t, x\in\er^d,\\
&v({t},x) = f(x),\,x\in \er^d.
\end{aligned}
\right.
\end{equation*}
Applying It\^o formula to $v(s,X^i_s)$ for $i=1,2$,
\begin{align*}
\int f u^i_t \,dx = \Ee \left[v(t,X^i_t)\right] & = \Ee\left[ v(0,X_0)\right] + \Ee \left[\int_0^t  \frac{\partial v}{\partial s} (s,X^i_s) ds\right] +
\frac{1}{2}\EE \left[\int_0^t   \sigma^2(u^i_s(X_s^i))\triangle_xv(s,X^i_s) ds\right] \\
& = \Ee \left[v(0,X_0)\right] +
\frac{1}{2}\EE \left[\int_0^t   \left(\sigma^2(u^i_s(X_s^i))-\gamma^2\right)\triangle_xv(s,X^i_s) ds\right]
\end{align*}
so that for $f\in\Cc^\infty_c(\er^d)$, {$G^\gamma_{0,t}=G^\gamma_t$,}
 \[
 \int f u^i_t\,dx=\int f\left(G^\gamma_{t}(\mu_0)+\frac{1}{2}{\sum_{k=1}^d} \int_0^t{\frac{\partial^{2}}{\partial x_k^2}}G^\gamma_{t-s}\left(\left(\sigma^2(u^i_s)-\gamma^2\right)u^i_s\right)\,ds\right)\,dx.
 \]
 We then obtain that
 \[
 \int f\left(u^1_t-u^2_t\right)\,dx=\int f\left(\frac{1}{2}\sum_{k=1}^d \int_0^t{\frac{\partial^{2}}{\partial x_k^2}}G^\gamma_{t-s}\left(\left(\sigma^2(u^1_s)-\gamma^2\right)u^1_s-\left(\sigma^2(u^2_s)-\gamma^2\right)u^2_s\right)\,ds\right)\,dx.
 \]
Next, we take the supremum over all $f\in\Cc^\infty_c(\er^d)$ such that $\Vert f\Vert_{L^2(\er^d)}=1$ and we integrate the resulting expression over $(0,T)$. It follows that
 \begin{equation}\label{proofst:mode_weak_1}
  \Vert u^1-u^2\Vert_{L^2((0,{T})\times\er^d)}=\Vert\frac{1}{2}\sum_{k=1}^d\int_0^{.}{\frac{\partial^{2}}{\partial x_k^2}}G^\gamma_{.-s}\left(
  \left(\sigma^2(u^1_s)-\gamma^2\right)u^1_s-\left(\sigma^2(u^2_s)-\gamma^2\right)u^2_s\right)\,ds\Vert_{L^2((0,{T})\times\er^d)}.
 \end{equation}
Now, let us recall that for all $f\in L^2((0,{T})\times\er^d)$, (see e.g. Stroock and Varadhan \cite{StrVar-79}, Appendix A.2, Lemmas A.2.1 and A.2.2)
 \[
 \Vert \int_0^{.}\partial^2_{x_kx_l}G^\gamma_{.-s}(f(s))\,ds\Vert_{L^2((0,{T})\times\er^d)}\leq \frac{2}{\gamma^2}\Vert f \Vert_{L^2((0,{T})\times\er^d)},\,1\leq k,l\leq d.
 \]
{Thank to} the boundedness of  $\sigma^2(u^i),\,i=1,2$, the preceding estimate ensures that the r.h.s. of \eqref{proofst:mode_weak_1} is well defined.
 A closer investigation of the proof arguments in  \cite{StrVar-79} enables to slightly improve the preceding estimate with
 \begin{equation}\label{eq:StroockVaradhan-estim}
 \begin{aligned}
 \Vert \int_0^{.}\triangle_x G^\gamma_{.-s}(f(s))\,ds\Vert_{L^2((0,{T})\times\er^d)}&=\Vert \sum_{k=1}^d\int_0^{.}{\frac{\partial^{2}}{\partial x_k^2}}G^\gamma_{.-s}(f(s))\,ds\Vert_{L^2((0,{T})\times\er^d)}\leq \frac{2}{\gamma^2}\Vert f \Vert_{L^2((0,{T})\times\er^d)}
 \end{aligned}
 \end{equation}
(see the appendix section, for a short proof). Applying this estimate to \eqref{proofst:mode_weak_1}, we obtain
\begin{align*}
 \Vert u^1-u^2\Vert_{L^2((0,{T})\times\er^d)}&\leq \frac{1}{\gamma^2} \Vert \left(\sigma^2(u^1)-\gamma^2\right)u^1-\left(\sigma^2(u^2)-\gamma^2\right)u^2\Vert_{L^2((0,{T})\times\er^d)}\\
&\qquad =  \frac{1}{\gamma^2}\Vert \left( \sigma^2(u^1)u^1-\sigma^2(u^2)u^2-\gamma^2(u^1-u^2)\right)\Vert_{L^2((0,{T})\times\er^d)}.
\end{align*}
For $\alpha(r)$ as in \eqref{def:alpha}, by observing that the first order Taylor expansion writes
$$
\sigma^2(r_2)r_2-\sigma^2(r_1)r_1=\int_0^1(r_1-r_2)\alpha(r_1+\theta(r_2-r_1))\,d\theta,
$$
{and choosing $\gamma>0$ so that
$$
\gamma^2>\max_{i=1,2}|\int_0^1\alpha(u^i+\theta(u^i-u^j))\,d\theta|,
$$}
we get
\begin{align*}
|\left(\sigma^2(u^1)u^1-\sigma^2(u^2)u^2-\gamma^2(u^1-u^2)\right)|
&=\left(\gamma^2-\int_0^1\alpha(u_1+\theta(u_2-u_1))\,d\theta\right) \, |u^1-u^2|,
\end{align*}
and deduce that
\begin{equation}\label{proofst:uniquenessmod_1}
\Vert u^1-u^2\Vert^2_{L^2((0,{T})\times\er^d)}\leq \frac{1}{\gamma^4}\Vert\left(\gamma^2-\int_0^1\alpha(u_1+\theta(u_2-u_1))\,d\theta\right) \left|u^1-u^2\right|\Vert^2_{L^2((0,{T})\times\er^d)}.
\end{equation}
Splitting $\Vert u^1-u^2\Vert^2_{L^2((0,{T})\times\er^d)}$ into the sum
$$
\Vert (u^1-u^2)\1_{\{u^2-u^1\geq \kappa\}}\Vert^2_{L^2((0,{T})\times\er^d)}+\Vert (u^1-u^2)\1_{\{u^2-u^1< \kappa\}}\Vert^2_{L^2((0,{T})\times\er^d)}
$$
for some arbitrary $\kappa>0$, \eqref{proofst:uniquenessmod_1} reduces to
\begin{align*}
&\Vert (u^1-u^2)\1_{\{u^2-u^1\geq \kappa\}}\Vert^2_{L^2((0,{T})\times\er^d)}\\
&\leq \frac{1}{\gamma^4}\Vert\left(\gamma^2-\int_0^1\alpha(u_1+\theta(u_2-u_1))\,d\theta\right) \left|u^1-u^2\right|\Vert^2_{L^2((0,{T})\times\er^d)}
-\Vert (u^1-u^2)\1_{\{u^2-u^1< \kappa\}}\Vert^2_{L^2((0,{T})\times\er^d)}\\
&\leq \frac{1}{\gamma^4}\Vert\left(\gamma^2-\int_0^1\alpha(u_1+\theta(u_2-u_1))\,d\theta\right) \left|u^1-u^2\right|\1_{\{u^2-u^1\geq \kappa\}}\Vert^2_{L^2((0,{T})\times\er^d)}
\end{align*}

Fixing $\kappa>0$, and setting
$$
\zeta^2{(\kappa)}:={\sup\big\{\beta>0\,;\int_0^1\alpha(r\kappa)\,dr>\beta\big\}}
$$
which is (strictly) positive by {\hypmodiiweaken and the monotone assumption of $\alpha$},
$$
\gamma^2-\int_0^1\alpha(u_1+\theta(u_2-u_1))\,d\theta{\leq \gamma^2-\int_0^1\alpha(\theta \kappa)\,d\theta}\leq \gamma^2-\zeta^2{(\kappa)},
$$
which implies that $\Vert |u^1-u^2|\ind_{\{u^2-u^1\geq \kappa\}} \Vert_{L^2((0,{T})\times\er^d)}<(\gamma^2-\zeta^2)\Vert|u^1-u^2|\ind_{\{u^2-u^1\geq \kappa\}}\Vert_{L^2((0,{T})\times\er^d)}/\gamma^2<\Vert|u^1-u^2|\ind_{\{u^2-u^1\geq \kappa\}}\Vert_{L^2((0,{T})\times\er^d)}$. Since $\kappa$ is arbitrary, $u_2\leq u_1$ for a.e. on $(0,{T})\times\er^d$. By symmetry, we can also exchange $u^1$ and $u^2$, and deduce that $u_1\leq u_2$ for a.e. on $(0,{T})\times\er^d$, from which we conclude that $u_1=u_2$ for a.a. $(t,x)\in(0,{T})\times\er^d$.
This conclude the claim.
\end{proof}

To conclude on the strong uniqueness of the solution to \eqref{prop:limit_nonlinear_smooth_sdepde}, let us recall the following result due to Champagnat and Jabin~\cite{ChaJab-17}.
{
\begin{theorem}[Theorems 1.1 and 1.2, \cite{ChaJab-17}]\label{thm:ChJab-17} Let $(Z^1_t;\,t\geq 0)$ and $(Z^2_t;\,t\geq 0)$ be two solutions to the SDE
\[
dZ_t=\Sigma(t,Z_t)\,dW_t,\,Z_0=\xi,
\]
with one-dimensional time marginal $u_{Z_1}(t,z)\,dz$ and $u_{Z_2}(t,z)\,dz$ in
$L^{2q}_{\loc}((0,\infty);W^{1,2p}(\er^d))$. If $\Sigma:(0,\infty)\times\er^d\rightarrow{\er^{d\times d}}$
 is in $L^\infty((0,\infty)\times\er^d)\cap L^{q'}_{\loc}((0,\infty);W^{1,p'}(\er^d))$,
  for $1/p+1/{p'}=1$, $1/q+1/{q'}=1$, then one has pathwise uniqueness: a.s. $\sup_{t\geq 0}|Z^1_t-Z^2_t|=0$.
\end{theorem}}

According to Proposition \ref{prop:moderateduniq} and Corollary \ref{coro:limit_nonlinear_smooth_sdepde}, for any solution to \eqref{eq:moderated}, $u$ is in $L^\infty((0,T)\times\er^d)$. Hence, a direct application of Theorem \ref{thm:ChJab-17} (with $p=q=1,\,p'=q'=\infty$) gives

\begin{proposition}\label{prop:moderatedstronguniq} {Under \hypmodo, \hypmodi and \hypmodii,} \eqref{eq:moderated} admits at most one {strong} solution.
\end{proposition}

\subsection{Generalization to matrix valued diffusion}\label{rmk:ModerateGeneral}

\newcommand{\hypmodibis}{{\textit(${\textit A}^\prime_1$})\xspace}
\newcommand{\hypmodiii}{{\textit(${\textit A}^\prime_2$})\xspace}
\newcommand{\hypmodiiiweaken}{{\textit(${\textit A}^\prime_2$-weakened})\xspace}

\newcommand{\Am}{{a}}

We may remark that the main ideas for the proof of Theorem \ref{thm:moderatedexist} can be extended to obtain the existence and uniqueness of a weak solution to \eqref{eq:moderated} in the situation where the diffusion component is a $d\times d$-matrix valued function; namely $\sigma:[0,\infty)\rightarrow \er^{d\times d}$. Such extension holds provided that the assumption \hypmodo remains unchanged meanwhile \hypmodi, \hypmodii and \hypmodiiweaken are respectively replaced by the following:
\begin{hypothesis}
\item{\bf \hypmodibis} The map  $r \mapsto \sigma(r)\in\er^{d\times d}$ is continuously differentiable on $\er^+$.
\item{\bf\hypmodiii} For $\Am(r)=\sigma\sigma^*(r)$, the map $r\in[0,\infty)\mapsto \alpha(r)\in\er^{d\times d}$
given by
\[
\alpha(r)=\{\alpha^{i,j}(r)  := (\Am^{i,j}(r) r)^\prime = (\Am^{i,j})^\prime (r) r  + \Am^{i,j}(r),\,1\leq i,j\leq d\},
\]
is continuous  and strongly elliptic in the sense that, for some $\eta_\Am>0$,
\[
\xi \cdot\alpha(r)\xi \geq \eta_\Am|\xi|^2,\,\forall\,r\geq 0,\,\forall\,\xi\in\er^d.
\]
\item{\bf \hypmodiiiweaken} For $\Am(r)=\sigma\sigma^*(r)$, the map $r\in[0,\infty)\mapsto \alpha(r)\in\er^d\times\er^d$
given by
\[
\alpha(r)=\{\alpha^{i,j}(r)  := (\Am^{i,j}(r) r)^\prime = (\Am^{i,j})^\prime (r) r  + \Am^{i,j}(r),\,1\leq i,j\leq d\}
\]
is continuous and positive semi-definite:
\[
\xi \cdot \alpha(r)\xi \geq 0,\,\forall\xi\in\er^d.
\]
\end{hypothesis}
The particular strict monotone assumption in Proposition \ref{prop:moderateduniq} can be replaced by the assumption that
\begin{equation}\label{eq:MonotoneMultidim}
\xi \cdot \alpha(r)\xi < \xi\cdot \alpha(r')\xi,\,\forall\xi\in\er^d,\,\forall r,r'\geq 0\,\text{such that}\,r<r'.
\end{equation}
Under \hypmodo, \hypmodibis and \hypmodiii,  Theorem \ref{thm:moderatedexist} can be extended to the existence and uniqueness of a strong solution to
\begin{equation}\label{eq:moderatedmatrixdiff}
\left\{
\begin{aligned}
&X_t =  X_0 + \int_0^t \sigma ( u(s,X_s) ) dW_s, \quad 0\leq t \leq T,\\
&{d}\law(X_t)=u(t,x) dx\,\mbox{ with }\,u \in {L^\infty((0,T)\times\er^d)\cap~ }L^2((0,T){\times}\er^d),\\
&u(0{,x})=u_0{(x)},\,{x\in\er^d},\,\,\sigma:\er^+\rightarrow \er^{d\times d}.
\end{aligned}
\right.
\end{equation}

\subsubsection*{For the existence of a weak solution to \eqref{eq:moderatedmatrixdiff}}
Assumptions \hypmodo, \hypmodibis and \hypmodiii are enough to replicate the proof arguments of the estimates  \ref{eq:smoothednonlinear_ineq_Linfini} and \ref{eq:smoothednonlinear_ineq_energy} in Lemma \ref{lem:WellposednessSmoothpde} and enables to construct, as in Proposition \ref{prop:WellposednessSmoothsde}, a weak solution to
\begin{equation*}
\left\{
\begin{aligned}
&X^{\varepsilon}_t =  X_0 + \int_0^t \sqrt{\Am_{\varepsilon}( u^{\varepsilon}(s,X_s) )} dW_s, \quad 0\leq t \leq T,\\
&{d}\;\!\law(X^{\varepsilon}_t)=u^{\varepsilon}(t,x) dx,\mbox{ with }  u^{\varepsilon} \in {L^\infty((0,T)\times\er^d)\cap~ }L^2((0,T){\times}\er^d),\\
&u^{\varepsilon}({0,x})=u_0{(x)},\,\,\sigma:\er^+\rightarrow \er^d\times\er^d,
\end{aligned}
\right.
\end{equation*}
where $\sqrt{\Am_{\varepsilon}(r)}$ is the square root matrix of $\Am(r)+{\varepsilon} I_d$, {$\Am'(r)=\{(\Am^{i,j})'(r),\,1\leq i,j\leq d\}$}.  Although the identity \eqref{eq:Vasqueztype_ineq} doesn't have any trivial multidimensional extension, {and since} $u^{\varepsilon}$ satisfies the analogous of \eqref{eq:mode_pde_linear}: for all
for all $0\leq T_0\leq T$, $f\in\Cc^{1,2}_c([0,T_0]\times\er^d)$,
\begin{equation}\label{eq:modepdenonlineardiffmatrix}
\begin{aligned}
&\int_{\er^d} u^{\varepsilon}(T_0,x)f(T_0,x)\,dx-\int_{\er^d}u_0(x)f(0,x)\,dx\\
&=\int_{(0,T_0)\times\er^d}u^{\varepsilon}(t,x)\partial_t f(t,x)-\nabla_x u^{\varepsilon}(t,x)\cdot  \Am'_{\varepsilon}({u^{\varepsilon}}(t,x))\nabla_x f(t,x) \,dt\,dx,
\end{aligned}
\end{equation}
the convergence (up to a subsequence) of $\{(X^{\varepsilon}_t,u^{\varepsilon}(t);\,0\leq t\leq T)\}_{{\varepsilon}>0}$ to a weak solution to \eqref{eq:moderatedmatrixdiff} can still be derived from the energy estimate:
\[
\max_{0\leq t\leq T}\Vert u^{\varepsilon}(t)\Vert^2_{L^2(\er^d)}+\eta_\Am\int_0^T\Vert\nabla_x u^{\varepsilon}(t)\Vert^2_{L^2(\er^d)}\,dt\leq \Vert u_0\Vert^2_{L^2(\er^d)},
\]
which follows from \eqref{eq:modepdenonlineardiffmatrix}.

\subsubsection*{For the uniqueness of a strong solution to \eqref{eq:moderatedmatrixdiff}}
Under \hypmodibis and \hypmodiiiweaken, the main arguments of Proposition \ref{prop:moderateduniq} can be extended, replacing $G^\gamma_{s,t}$ by the fundamental solution to $G^\Gamma_{s,t}$ related to the parabolic operator
$L(f)= \partial_sf +\frac{1}{2}\trace\left(\Gamma \nabla^2_xf\right)$
where $\Gamma$ is a (constant) positive definite matrix such that $\xi\cdot \Gamma \xi\geq \gamma^2|\xi|^2, \gamma\neq 0$. Taking two weak solutions $(X^1_t,\,u^1_t;\,t\in[0,T])$ and $(X^2_t,\,u^2_t;\,t\in[0,T])$ to \eqref{eq:modepdenonlineardiffmatrix}, and by replicating the first proof steps of Proposition \ref{prop:moderateduniq}, we get
the analog of \eqref{proofst:mode_weak_1}:
\begin{equation*}
  \Vert u^1-u^2\Vert_{L^2((0,{T})\times\er^d)}=\Vert\frac{1}{2}\sum_{k,l=1}^d\int_0^{.}\partial^2_{x_kx_l}G^\gamma_{.-s}\left(
  \left(\Am(u^1_s)-\Gamma\right)^{k,l}u^1_s-\left(\Am(u^2_s)-\Gamma\right)^{k,l}u^2_s\right)\,ds
  \Vert_{L^2((0,{T})\times\er^d)}.
 \end{equation*}
 Using the following generalization of \eqref{eq:StroockVaradhan-estim} (see Appendix section):
\[
 \Vert \sum_{i,j=1}^d \int_0^{.}\partial^2_{x_ix_j} G^\Gamma_{.-s}(F^{i,j}(s))\,ds\Vert_{L^2((0,{T})\times\er^d)}\leq \frac{2}{\gamma^2}\sum_{i,j=1}^d \Vert F^{i,j} \Vert_{L^2((0,{T})\times\er^d)},\,\,F^{i,j}\in L^2((0,{T})\times\er^d),
\]
 we deduce the analog of \eqref{proofst:uniquenessmod_1}
\begin{equation*}
\Vert u^1-u^2\Vert_{L^2((0,{T})\times\er^d)}\leq \frac{1}{\gamma^2}\sum_{k,l=1}^d\Vert\left(\Gamma-\int_0^1\alpha(u_1+\theta(u_2-u_1))\,d\theta\right)^{k,l} \left|u^1-u^2\right|\Vert_{L^2((0,{T})\times\er^d)}.
\end{equation*}
 Taking $\Gamma$ large enough so that $\xi\cdot( \Gamma-\alpha(r))\xi < 0$, for all $\xi\in\er^d$, $r\geq 0$,  the strict elliptic assumption in \hypmodiii or \eqref{eq:MonotoneMultidim} are sufficient to ensure that $u^1=u^2$.
The uniqueness of a strong solution to \eqref{eq:moderatedmatrixdiff} still follows from Champagnat and Jabin \cite{ChaJab-17}.

\section{Conditional nonlinear diffusion case}\label{sec:ConditionalDiffusion}

\begin{hypothesis}\label{hypo:weak}
\item{\hypcondo} The initial law $\mu_0$ admits a density $\rho_0$ such that $\int_{\er^d\times\er^d} (|x|^2+|y|^2)\rho_0(x,y)\,dx\,dy<\infty$.
\item{\hypcondi} The coefficients $b$ and $\sigma$ are bounded Lipschitz continuous functions
\item{\hypcondii} The kernels $\ell$ and $\gamma$ are bounded and Lipschitz continuous on $\er^d$.
\item{\hypcondiii} Strong ellipticity is assumed for $\sigma$: there exists $a_*>0$ such that, for all $x\in \er^d$,
    \[ a_* |\xi|^2 \leq \xi \sigma(x) \xi,\quad \forall\,\xi\in\er^d. \]
\end{hypothesis}

\begin{hypothesis}\label{hypo:strong}
\item{\hypcondquatre }
The initial marginal density $\rho_X(0,x) = \int_{\er^d}  \rho_0(x,y) dy$ is in $L^1(\er^d)\cap L^p(\er^d)$ for some  $p\geq 2d+2$. \\
Moreover, for all $R>0$, for all $x\in B(0,R)$, there exists a constant $\mu_R>0$ such that $\rho_X(0,x)\geq \mu_R$.

\item{\hypcondcinq } $\sigma$ and $\gamma$ are in $\Cc^2(\er^d)$ with bounded derivatives up to second order.
\item{\hypcondsix } Strong ellipticity is assumed for $\gamma\gamma^*$: there exists $\alpha_*>0$ such that, for all $(x,y)\in \er^d\times\er^d$,
   \[ \alpha_* |\xi|^2  \leq  \xi \gamma(y) \gamma(y)^* \xi,\quad\forall\,\xi\in\er^d. \]
\end{hypothesis}

Our main result concerns the wellposedness (in the weak and strong sense) of a solution to \eqref{eq:NonlinearLangevin}. More precisely, we have
\begin{theorem}\label{thm:wellposedNonlinearLangevin}
Under Hypothesis \ref{hypo:weak}, there exists a unique weak solution to \eqref{eq:NonlinearLangevin}.
With the addition of Hypothesis  \ref{hypo:strong}, pathwise uniqueness holds for the solution of
 \eqref{eq:NonlinearLangevin} {and $\law(X_t,Y_t)$ admits a density function at all time $0\leq t\leq T$}.
\end{theorem}

Before entering in the details of the proof, let us point out an important remark for the construction of the solution of \eqref{eq:NonlinearLangevin}. Consider for a while the case when $(X_t;\,t\geq 0)$ doesn't depend on $(Y_t;\,t\geq 0)$, namely when $b$ does not depend on $Y$, or the simpler situation when $b=0$. Hypothesis \ref{hypo:weak}  ensure the existence of a unique strong solution for
\begin{equation}\label{proofst:conditionalexists_0}
X_t=X_0+\int_0^t \sigma(X_s)\,dB_s.
\end{equation}
Then, based on the fact that $(X_t;\,t\geq 0)$ is now an exogenous process, we can consider the following fixed point construction, similar to those in Sznitman \cite{Sznitman-89} and M\'el\'eard \cite{Meleard-96}. From now on, we fix  an arbitrary time horizon $0\leq T<+\infty$ and we consider the classical Hilbert space $M^2(0,T)$ of real adapted continuous processes $\zeta = (\zeta_t; 0\leq t\leq T)$ such that $\Ep[\int_0^T \zeta^2_s ds]<+\infty$ ($\Ep$ denoting the expectation under $\PP$),  and endowed with the following scalar product and norm
\[(\zeta,\xi)_{{c}} = \Ep\left[\int _0^T {\exp(-c s)}\zeta_s \xi_s ds\right],\quad \|\zeta \|_c^2 = \Ep\left[\int_0^T  \exp(-c s) |\zeta_s|^2 ds\right],\]
where $c$ is a positive constant that will be chosen later.
Given $(\Omega,\FF,(\Ff_t;\,0\leq t\leq T),\Pp)$ a filtered probability space under which are defined $(W_t;\,0\leq t\leq T)$ and $(B_t;\,0\leq t\leq T)$, two independent $\er^d$-Brownian motions, and $(X_0,Y_0)\sim{\mu}_0$ (independent of $(W_t;\,0\leq t\leq T)$ and $(B_t;\,0\leq t\leq T)$).

From any element $\zeta$ in $M^2(0,T)$, we construct the application $\zeta \mapsto Y(\zeta)$ {taking values} in $M^2(0,T)$ and defined as
\begin{equation}\label{proofst:conditionalexists_1}
Y(\zeta)_t = Y_0 + \int_0^t \Ep\left[\ell(\zeta_s)\,|\,X_s\right] ds+\int_0^t \Ep\left[\gamma(\zeta_s)\,|\,X_s\right] dW_s,\,0\leq t\leq T.
\end{equation}

Owing to \hypcondo and \hypcondii, it is clear that $\|Y(\zeta)\|_c<+\infty$. Now, for $\zeta$ and $\xi$ in $M^2(0,T)$, $0\leq t\leq T$, we have
\begin{align*}
\Ep\left[|Y(\zeta)_t - Y(\xi)_t|^2\right]&\leq2 \int_0^{t}\Ep\left[\big|\Ee_\Pp\left[\ell(\zeta_s)\,|\,X_s\right]
-\Ep\left[\ell(\xi_s)\,|\,X_s\right]\big|^2\right]\,ds\\
& \quad +2 \int_0^{t}\Ee_\Pp\left[\left\|
\Ee_\Pp\left[\gamma(\zeta_s)\,|\,X_s\right]-\Ee_\Pp\left[\gamma(\xi_s)\,|\,X_s\right]
\right\|^2\right]\,ds\\
&\leq 2 \left(\Vert \ell\Vert_{\textit{Lip}}^2+\Vert \gamma\Vert_{Lip}^2\right)\int_0^{t}\Ee_\Pp\left[\Ee_\Pp\left[|\zeta_s - \xi_s|^2\,|\,X_s\right]\right]\,ds\\
& \leq 2 \left(\Vert \ell\Vert_{Lip}^2+\Vert \gamma\Vert_{Lip}^2\right)\int_0^{t}\Ee_\Pp\left[|\zeta_s - \xi_s|^2\right]\,ds.
\end{align*}
Multiplying both sides by $\exp(-ct)$ and integrating (in time) the resulting expression over the interval $(0,T)$ gives
\[
\int_0^T\exp(-ct) \Ep\left[|Y(\zeta)_t - Y(\xi)_t|^2\right]\,dt\leq 2 \left(\Vert \ell\Vert_{Lip}^2+\Vert \gamma\Vert_{Lip}^2\right)\int_0^T\exp(-ct)\int_0^{t}\Ee_\Pp\left[|\zeta_s - \xi_s|^2\right]\,ds\,dt.
\]
An integration by part in time then yields
\begin{align*}
\|Y(\zeta) - Y(\xi)\|_c^2 &\leq -\frac{\exp(-cT)}{c}\int_0^{T}\Ee_\Pp\left[|\zeta_s - \xi_s|^2\right]\,ds+\frac{2}{c} \left(\Vert \ell\Vert_{Lip}^2+\Vert \gamma\Vert_{Lip}^2\right)\int_0^T\exp(-ct)\Ee_\Pp\left[|\zeta_t - \xi_t|^2\right]\,dt\\
&\leq  \frac{2}{c} \left(\Vert \ell\Vert_{Lip}^2+\Vert \gamma\Vert_{Lip}^2\right) \|\zeta - \xi\|^2_c.
\end{align*}
Choosing $c > 2 (\Vert \ell\Vert_{Lip}^2+\Vert \gamma\Vert_{Lip}^2)$, we get the existence of a unique fixed point solution of equation \eqref{eq:NonlinearLangevin}, when $b=0$. In the case $b(x,y)=b(x)$, the same arguments lead to the same result.
\medskip

The rest of this section is dedicated to the proof of
 Theorem \ref{thm:wellposedNonlinearLangevin}, which essentially relies on a Girsanov transform to go back to a situation similar to the previous case.
First, in Section \ref{sec:girsanov}, we collect some preliminary remarks on the Girsanov transform that we use to remove the drift in \eqref{eq:NonlinearLangevin}, and deduce some apriori controls on the associated change of probability measure.

Second, in Section \ref{sec:weak}, we use the fixed point technique for the $L^2$-existence and weak uniqueness for solution to \eqref{eq:NonlinearLangevin}.

Finally in {S}ection \ref{sec:strong}, assuming some stronger regularity on the kernels $\ell$ and $\gamma$ and strong ellipticity on $\gamma$ (see Hypothesis  \ref{hypo:strong}, we obtain some apriori regularity on the nonlinear coefficients using averaging lemma technique, and then deduce the strong uniqueness property.

\subsection{Preliminary remarks on \eqref{eq:NonlinearLangevin}}\label{sec:girsanov}
Fix an arbitrary  $0\leq T <+\infty$  and  let $(X_t,Y_t;\,0\leq t\leq T)$ be a solution to \eqref{eq:NonlinearLangevin} up to $T$, defined on $(\Omega,\Ff,(\Ff_t;\,0\leq t\leq T),\PP)$. Then, define $(Z_t; 0\leq t\leq T)$ as
$$
Z_t=\exp\left\{-\int_0^t\left(\sigma^{-1} b\right)(X_s,Y_s)\,dB_s-\frac{1}{2}\int_0^t \left|\sigma^{-1} b\right|^2(X_s,Y_s)\,ds \right\},
$$
for $\sigma^{-1}(x)$ the inverse matrix of $\sigma(x)$, $(\sigma^{-1}b)(x,y)=\sigma^{-1}(x)b(x,y)$ {and 
$\left|\sigma^{-1} b\right|^2(x,y)=\left(\sigma^{-1} b\right)(x,y)\cdot\left(\sigma^{-1} b\right)(x,y)$}.
Then, under the probability measure $\QQ$ defined on $(\Omega,\Ff_T,(\Ff_t;\,0\leq t\leq T))$ by
$$
\frac{d\QQ}{d\PP}\Big|_{\Ff_T}=Z_T,
$$
the process
$$
\Bmod_t=\int_0^t(\sigma^{-1}b)(X_s,Y_s)\,ds+B_t,\,0\leq t\leq T,
$$
is a $\er^d$-Brownian motion (by means of Girsanov transformation). Observing that the covariation between $(W_t;\,0\leq t\leq T)$ and $(\Bmod_t;\,0\leq t\leq T)$ is zero, $(\Bmod_t;\,0\leq t\leq T)$ is independent of $(W_t;\,0\leq t\leq T)$.

In addition, for all $\Ff_t$-adapted process $(\theta_t;\,0\leq t\leq T)$ such that $\Ep [|\theta_t|] < +\infty $ for all $t$, the characterization of the conditional expectation ensures that, $\PP$-a.s. (or equivalently $\QQ$-a.s.),
\begin{equation}\label{eq:WeightedConditionalExpect}
\Ep[\theta_t\,|\,X_t]=Z_t\Eq[(Z_t)^{-1}\theta_t\,|\,X_t],\,0\leq t\leq T.
\end{equation}
Following this change of probability measure, under $\QQ$, the equation \eqref{eq:NonlinearLangevin} formulate as the following self-contained SDE:
\begin{equation}
\label{eq:TransformedNonlinear}
\left\{
\begin{aligned}
&X_t=X_0+\int_0^t \sigma(X_s)\,d\Bmod_s,\,0\leq t\leq T,\\
&Y_t=Y_0+\int_0^tZ_s\Eq\left[Z^{-1}_s \ell(Y_s)\,|\,X_s\right]\,ds+\int_0^tZ_s\Eq\left[Z^{-1}_s \gamma(Y_s)\,|\,X_s\right]\,dW_s,\\
&Z_t=\exp\left\{-\int_0^t\left(\sigma^{-1} b\right)(X_s,Y_s)\,d\widehat{B}_s+\frac{1}{2}\int_0^t \left|\sigma^{-1} b\right|^2(X_s,Y_s)\,ds\right\},\\
&(X_0,Y_0)\sim \mu_0.
\end{aligned}
\right.
\end{equation}
{Conversely}, starting from \eqref{eq:TransformedNonlinear}, defined on $(\Omega,\Ff_T,(\Ff_t,\,0\leq t\leq T),\QQ)$ {endowed with}
two independent Brownian motions $(\Bmod_t;\,0\leq t\leq T)$ and $(W_t;\,0\leq t\leq T)$, independent to $(X_0,Y_0)$, one can easily check that $(X_t,Y_t;\,0\leq t\leq T)$ is a solution to \eqref{eq:NonlinearLangevin} on $(\Omega,\Ff_T,(\Ff_t,\,0\leq t\leq T),\widehat{\PP})$ where $\widehat{\PP}$ is given by
$$
\frac{d\widehat{\PP}}{d\QQ}\Big|_{\Ff_T}=Z^{-1}_T=\exp\left\{\int_0^t\left(\sigma^{-1} b\right)(X_s,Y_s)\,d\Bmod_s-\frac{1}{2}\int_0^t \left|\sigma^{-1} b\right|^2(X_s,Y_s)
\,ds\right\},\,0\leq t\leq T.
$$

The existence and uniqueness of a weak solution to \eqref{eq:NonlinearLangevin} is then an immediate consequence of the existence and uniqueness of a weak solution to \eqref{eq:TransformedNonlinear}.

Let $\zeta$ in $M^2(0,T)$. We consider the linearized system
\begin{equation}
\label{eq:Transformedlinear_bis}
\left\{
\begin{aligned}
&X_t=X_0+\int_0^t \sigma(X_s)\,d\Bmod_s,\,0\leq t\leq T,\\
&Z(\zeta)_t=\exp\left\{-\int_0^t\left(\sigma^{-1} b\right)(X_s,\zeta_s)\,d\widehat{B}_s+\frac{1}{2}\int_0^t \left|\sigma^{-1} b\right|^2(X_s,\zeta_s) ds\right\},\\
&(X_0,Y_0) \sim \mu_0.
\end{aligned}
\right.
\end{equation}

\begin{lemma}\label{lem:ExpMartPrelim}
Assume Hypothesis \ref{hypo:weak}. Let $\zeta$  and $\xi$ in $M^2(0,T)$.
\begin{itemize}
\item[(i)] There exits a positive constant $C$ depending only on $T$ and on $\|\sigma^{-1}b\|_{\infty}$ such that
\[\Eq \big [Z(\zeta)^2_t | X_t] (Z(\zeta)_t)^{-2} \leq  C.\]
\item[(ii)]  There exits a positive constant $C'$ depending only on $T$,  $\|\sigma^{-1}b\|_{\infty}$ and $\|\ell\|_{Lip} + \|\gamma\|_{Lip} $ such that
\begin{align}
\Eq\left[ \max_{0\leq s\leq t} \frac{|Z(\zeta)^2_s - Z(\xi)_s|^2}{Z(\zeta)^2_s} \right] \leq C' \int_0^t
\Eq\left[ \frac{|Z(\zeta)_s - Z(\xi)_s|^2}{Z(\zeta)^2_s} +  |\zeta_s -\xi_s|^2 \right] ds.
\end{align}
\end{itemize}
\end{lemma}
\begin{proof}
For \textit{(i)}, we fix  $\zeta\in M^2(0,T)$. We denote by ${\EEE}$ the exponential martingale {(under $\QQ$)} defined as
\[\EEE_t = \exp\left\{-\int_0^t 2 \left(\sigma^{-1} b\right)(X_s,\zeta_s)\,d\widehat{B}_s{-} \int_0^t 2 \left|\sigma^{-1} b\right|^2(X_s,\zeta_s) ds\right\}.\]
From \eqref{eq:Transformedlinear_bis},  we have
\[Z(\zeta)^2_t=\EEE_t \exp\left\{{+3} \int_0^t \left|\sigma^{-1} b\right|^2(X_s,\zeta_s) ds\right\}\]
and then
\begin{align*}
\Eq \big [(Z(\zeta))^2_t | X_t\big ] (Z(\zeta)_t)^{-2}
&=\EEE^{-1}_t \Eq \left[\EEE_t\exp\left\{{+3}\int_0^t \left|\sigma^{-1} b\right|^2(X_s,\zeta_s) ds\right\}   | X_t\right] \exp\left\{ {-3}\int_0^t \left|\sigma^{-1} b\right|^2(X_s,\zeta_s) ds\right\}\\
&\quad \leq \EEE^{-1}_t \Eq \left[\EEE_t\exp\left\{{+3} \int_0^t \left|\sigma^{-1} b\right|^2(X_s,\zeta_s) ds\right\}   | X_t\right].
\end{align*}
Let us define the probability measure $\widetilde{\QQ}$ on $(\Omega,\Ff_T,(\Ff_t;\,0\leq t\leq T))$ by
$$
\frac{d\widetilde{\QQ}}{d\QQ}\,\Big|_{\Ff_T}=\EEE_T.
$$
Then, as in \eqref{eq:WeightedConditionalExpect}, we  obtain that for all $\Ff_t$-adapted process $(\theta_t;\,0\leq t\leq T)$ such that $\Ez [|\theta_t|] < +\infty $ for all $t$ in $[0,T]$, the characterization of the conditional expectation ensures that, $\PP$-a.s. (or equivalently $\QQ$-a.s.),
\begin{align*}
\Ez[\theta_t\,|\,X_t]= \EEE^{-1}_t\Eq[\EEE_t\theta_t\,|\,X_t],\,0\leq t\leq T,
\end{align*}
from which we immediately deduce that

\[\EEE^{-1}_t \Eq \left[\EEE_t\exp\left\{{+3} \int_0^t \left|\sigma^{-1} b\right|^2(X_s,\zeta_s) ds\right\}   | X_t\right]\leq \exp({3}T \|\sigma^{-1} b\|_{\infty}).\]

For \textit{(ii)}, we fix again  a $\zeta\in M^2(0,T)$. From \eqref{eq:Transformedlinear_bis},  we have
\begin{align*}
dZ(\zeta)_t = Z(\zeta)_t \left( - (\sigma^{-1} b)(X_t,\zeta_t)\,d\widehat{B}_t  + |\sigma^{-1} b|^2(X_t,\zeta_t)\,dt\right),\,Z(\zeta)_0=1,
\end{align*}
and
\begin{align*}
dZ(\zeta)^{-1}_t = Z(\zeta)^{-1}_t  (\sigma^{-1} b)(X_t,\zeta_t)\,d\widehat{B}_t,\,Z(\zeta)_0=1,
\end{align*}
from which we compute, using the It{\^o} formula, for a   $\xi\in M^2(0,T)$
\begin{align*}
& |Z(\zeta)_t - Z(\xi)_t|^2  \\
& =   2 \int_0^t (Z(\zeta)_s - Z(\xi)_s)
\left\{ Z(\zeta)_s(\sigma^{-1} b)(X_s,\zeta_s)
- Z(\xi)_s (\sigma^{-1} b)(X_s,\xi_s)\right\}\,d\widehat{B}_s  \\
& \quad +  2 \int_0^t (Z(\zeta)_s - Z(\xi)_s)
\left\{Z(\zeta)_s |\sigma^{-1} b|^2(X_s,\zeta_s)  -
Z(\xi)_s|\sigma^{-1} b|^2(X_s,\xi_s)\right\} ds  \\
& \quad +  \int_0^t
\left(Z(\zeta)_s(\sigma^{-1} b)(X_s,\zeta_s)- (Z(\xi)_s\sigma^{-1} b)(X_s,\xi_s)
 \right)^2ds
\end{align*}
and
\begin{align*}
d (Z(\zeta)^{-1}_t)^2  =   (Z(\zeta)^{-1}_t)^2 \left( 2   (\sigma^{-1} b)(X_t,\zeta_t)\,d\widehat{B}_t
 +  |\sigma^{-1} b|^2 (X_t,\zeta_t) dt \right).
\end{align*}
Applying again the  It{\^o} formula,
\begin{align*}
& \frac{|Z(\zeta)_t - Z(\xi)_t|^2}{Z(\zeta)^{2}_t}  \\
& = \int_0^t \frac{|Z(\zeta)_s - Z(\xi)_s|^2}{Z(\zeta)^{2}_s} \left( 2   (\sigma^{-1} b)(X_s,\zeta_s)\,d\widehat{B}_s
 + |\sigma^{-1} b|^2 (X_s,\zeta_s) ds \right)\\
&\quad +  2 \int_0^t \frac{(Z(\zeta)_s - Z(\xi)_s)}{Z(\zeta)^{2}_s}
\left\{ Z(\zeta)_s(\sigma^{-1} b)(X_s,\zeta_s)
- Z(\xi)_s (\sigma^{-1} b)(X_s,\xi_s)\right\}\,d\widehat{B}_s  \\
& \quad +  2 \int_0^t \frac{(Z(\zeta)_s - Z(\xi)_s)}{Z(\zeta)^{2}_s}
\left\{Z(\zeta)_s |\sigma^{-1} b|^2(X_s,\zeta_s)  -
Z(\xi)_s|\sigma^{-1} b|^2(X_s,\xi_s)\right\} ds  \\
& \quad  +  \int_0^t \frac{1}{Z(\zeta)^{2}_s}
\left|Z(\zeta)_s(\sigma^{-1} b)(X_s,\zeta_s)- (Z(\xi)_s\sigma^{-1} b)(X_s,\xi_s) \right|^2
ds  \\
& \quad  +
  4 \int_0^t \frac{(Z(\zeta)_s - Z(\xi)_s)}{Z(\zeta)^{2}_s}
\Big(   Z(\zeta)_s|\sigma^{-1} b|^2(X_t,\zeta_s) -
Z(\xi)_s (\sigma^{-1} b)(X_s,\xi_s)(\sigma^{-1} b)(X_s,\zeta_s)\Big) ds.
\end{align*}
$Z(\cdot)$ being an exponential martingale, the $L^2$-integrability of each integrands in the right-hand side of the preceding expression derive from the boundedness of $\sigma^{-1}b$.
For any terms of the form
\[\frac{1}{Z(\zeta)^{2}_s} \left\{Z(\zeta)_s g(X_s,\zeta_s)  -
Z(\xi)_s g (X_s,\xi_s)\right\}, \]
for $g$ equal to   $\sigma^{-1} b$  or $|\sigma^{-1} b|^2$, we add and subtract the same element $Z(\zeta)_s g(X_s,\zeta_s)$ to get
\begin{align*}
\frac{1}{Z(\zeta)_s} \Big( g(X_s,\zeta_s)  -  g(X_s,\xi_s) \Big)  + \frac{1}{Z(\zeta)^{2}_s} (Z(\zeta)_s - Z(\zeta)_s)  g(X_s,\xi_s).
\end{align*}
Noticing that $\sigma^{-1} b$ and $|\sigma^{-1} b|^2$ are bounded Lipschitz, by taking the expectation, and  by introducing the appropriate pivots in the three last integrals, we get
\begin{align*}
\Eq\left[\frac{|Z(\zeta)_t - Z(\xi)_t|^2}{Z(\zeta)^{2}_t}  \right]
& \leq C \|\sigma^{-1} b \|^2_{\infty}\int_0^t  \Eq\left[ \frac{|Z(\zeta)_s - Z(\xi)_s|^2}{Z(\zeta)^{2}_s}\right]
 ds \\
 &\quad  + C \min\left(\|\sigma^{-1} b \|_{Lip},2 \|\sigma^{-1} b \|_{\infty}\right)  \int_0^t  \Eq\left[\frac{|Z(\zeta)_s - Z(\xi)_s||\zeta_s - \xi_s|}{Z(\zeta)_s}\right] ds.
\end{align*}
We end the proof of \textit{(ii)} by applying Young's inequality.
\end{proof}

\subsection{$L^2$-existence and weak uniqueness}\label{sec:weak}

\begin{equation}
\label{eq:Transformedlinear}
\left\{
\begin{aligned}
&X_t=X_0+\int_0^t \sigma(X_s)\,d\Bmod_s,\,0\leq t\leq T,\\
&Y(\zeta)_t=Y_0+\int_0^tZ(\zeta)_s\Eq\left[Z(\zeta)^{-1}_s \ell(\zeta_s)\,|\,X_s\right]\,ds+\int_0^tZ(\zeta)_s\Eq\left[Z(\zeta)^{-1}_s \gamma(\zeta_s)\,|\,X_s\right]\,dW_s,\\
&Z(\zeta)_t=\exp\left\{-\int_0^t\left(\sigma^{-1} b\right)(X_s,\zeta_s)\,d\widehat{B}_s+\frac{1}{2}\int_0^t \left|\sigma^{-1} b\right|^2(X_s,\zeta_s)\,ds\right\},\\
&(X_0,Y_0)\sim \mu_0.
\end{aligned}
\right.
\end{equation}
Remark that, from  \eqref{eq:WeightedConditionalExpect}, for any bounded Borel function $g:\er^d\rightarrow\er$,
\begin{align}
Z(\zeta)_t\Eq[(Z(\zeta)_t)^{-1} |g(\zeta_t)|\,|\,X_t] = \Ep[|g(\zeta_t)| \,|\,X_t] \leq \|g\|_{L^\infty} ,\,0\leq t\leq T,
\end{align}
so that, $(Y(\zeta)_t;\,0\leq t\leq T)$ is in $M^2(0,T)$. In addition, we have
\begin{proposition}\label{prop:LagContraction} There exists $0<C<\infty$ depending only on $T$, $\Vert(\sigma^{-1}b)(\sigma^{-1}b)^*\Vert_{L^\infty}$, $\Vert\ell\Vert_{L^\infty}$ such that, for all $\zeta,\xi\in M^2(0,T)$, for all $0\leq t\leq T$,
\[
\Eq\left[|Y(\zeta)_t - Y(\xi)_t|^2\right]\leq C\int_0^t\Eq\left[\frac{|Z(\zeta)_s - Z(\xi)_s|^2}{|Z(\xi)_s|^2}+\frac{|Z(\zeta)_s - Z(\xi)_s|^2}{|Z(\zeta)_s|^2}+|\zeta_s-\xi_s|^2\right]\,ds.
\]
\end{proposition}
\begin{proof} Applying It\^o's formula, we get that
\begin{align*}
&|Y(\zeta)_t - Y(\xi)_t|^2\\
& =2\int_0^t\left(Y(\zeta)_s - Y(\xi)_s\right)\left(Z(\zeta)_s\Eq\left[Z(\zeta)^{-1}_s \ell(\zeta_s)\,|\,X_s\right]-Z(\xi)_s\Eq\left[Z(\xi)^{-1}_s \ell(\xi_s)\,|\,X_s\right]\right)\,ds\\
&\quad +\int_0^t\left(Z(\zeta)_s\Eq\left[Z(\zeta)^{-1}_s \gamma(\zeta_s)\,|\,X_s\right]-Z(\xi)_s\Eq\left[Z(\xi)^{-1}_s \gamma(\xi_s)\,|\,X_s\right]\right)\,dW_s\\
&\quad +\int_0^t\text{Trace}\left(\left(Z(\zeta)_s\Eq\left[Z(\zeta)^{-1}_s \gamma(\zeta_s)\,|\,X_s\right]-Z(\xi)_s\Eq\left[Z(\xi)^{-1}_s \gamma(\xi_s)\,|\,X_s\right]\right)\right.\\
&\quad\quad\quad\times\left.\left(Z(\zeta)_s\Eq\left[Z(\zeta)^{-1}_s \gamma(\zeta_s)\,|\,X_s\right]-Z(\xi)_s\Eq\left[Z(\xi)^{-1}_s \gamma(\xi_s)\,|\,X_s\right]\right)^*\right)\,ds.
\end{align*}
Taking the expectation on both sides of the preceding equality ($L^2$ integrability  is again {e}nsured from the boundedness of the coefficients combined with Identity  \eqref{eq:WeightedConditionalExpect}), it follows
\begin{align}\label{proofst:conditionalexists_2}
\begin{aligned}
&\Eq\left[|Y(\zeta)_t - Y(\xi)_t|^2\right]\\
&=2\int_0^t\Eq\left[\left(Y(\zeta)_s - Y(\xi)_s\right)\left(Z(\zeta)_s\Eq\left[Z(\zeta)^{-1}_s \ell(\zeta_s)\,|\,X_s\right]-Z(\xi)_s\Eq\left[Z(\xi)^{-1}_s \ell(\xi_s)\,|\,X_s\right]\right)\right]\,ds\\
&\quad +\int_0^t\Eq\left[\text{Trace}\left(\left(Z(\zeta)_s\Eq\left[Z(\zeta)^{-1}_s \gamma(\zeta_s)\,|\,X_s\right]-Z(\xi)_s\Eq\left[Z(\xi)^{-1}_s \gamma(\xi_s)\,|\,X_s\right]\right)\right.\right.\\
&\qquad\qquad\qquad\times\left.\left.\left(Z(\zeta)_s\Eq\left[Z(\zeta)^{-1}_s \gamma(\zeta_s)\,|\,X_s\right]-Z(\xi)_s\Eq\left[Z(\xi)^{-1}_s \gamma(\xi_s)\,|\,X_s\right]\right)^t\right)\right]\,ds.
\end{aligned}
\end{align}
By Young's inequality, for the first integral on the r.h.s., we have
\begin{align*}
&\int_0^t\Eq\left[\left(Y(\zeta)_s - Y(\xi)_s\right)\left(Z(\zeta)_s\Eq\left[Z(\zeta)^{-1}_s \ell(\zeta_s)\,|\,X_s\right]-Z(\xi)_s\Eq\left[Z(\xi)^{-1}_s \ell(\xi_s)\,|\,X_s\right]\right)\right]\,ds\\
&\leq \frac{1}{2}\int_0^t\Eq\left[\left|Y(\zeta)_s - Y(\xi)_s\right|^2\right]\,ds\\
& \quad +\frac{1}{2}\int_0^t
\Eq\left[\left|Z(\zeta)_s\Eq\left[Z(\zeta)^{-1}_s \ell(\zeta_s)\,|\,X_s\right]-Z(\xi)_s\Eq\left[Z(\xi)^{-1}_s \ell(\xi_s)\,|\,X_s\right]\right|^2\right]\,ds.
\end{align*}
In the last integral, adding and subtracting $Z(\xi)_s\Eq\left[Z(\zeta)^{-1}_s \ell(\zeta_s)\,|\,X_s\right]$ yields
\begin{align*}
  &\Eq\left[\left|Z(\zeta)_s\Eq\left[Z(\zeta)^{-1}_s \ell(\zeta_s)\,|\,X_s\right]-Z(\xi)_s\Eq\left[Z(\xi)^{-1}_s \ell(\xi_s)\,|\,X_s\right]\right|^2\right] \\
  &\leq 2 \Eq\left[\left(Z(\zeta)_s-Z(\xi)_s\right)^2\left(\Eq\left[Z(\zeta)^{-1}_s \ell(\zeta_s)\,|\,X_s\right]\right)^2\right]+
  2\Eq\left[Z(\zeta)^2_s\left|\Eq\left[Z(\xi)^{-1}_s \ell(\xi_s)-Z(\zeta)^{-1}_s \ell(\zeta_s)\,|\,X_s\right]\right|^2\right].
\end{align*}
The identity \eqref{eq:WeightedConditionalExpect} then  ensures that
\begin{align*}
\Eq\left[\left(Z(\zeta)_s-Z(\xi)_s\right)^2\left(\Eq\left[Z(\zeta)^{-1}_s \ell(\zeta_s)\,|\,X_s\right]\right)^2\right]
& =\Eq\left[\frac{\left(Z(\zeta)_s-Z(\xi)_s\right)^2}{Z(\zeta)^2_s}\left(Z(\zeta)_s\Eq\left[Z(\zeta)^{-1}_s \ell(\zeta_s)\,|\,X_s\right]\right)^2\right]\\
&\quad  \leq \Vert\ell\Vert_{L^\infty} \Eq\left[\frac{\left(Z(\zeta)_s-Z(\xi)_s\right)^2}{Z(\zeta)^2_s}\right].
\end{align*}
Using the properties of the conditional expectation, we also have
\begin{align}\label{eq:aux_pivot1}
\begin{aligned}
&\Eq\left[Z(\zeta)^2_s\left|\Eq\left[Z(\xi)^{-1}_s \ell(\xi_s)-Z(\zeta)^{-1}_s \ell(\zeta_s)\,|\,X_s\right]\right|^2\right]\\
&\leq \Eq\left[Z(\zeta)^2_s\Eq\left[\left|Z(\xi)^{-1}_s \ell(\xi_s)-Z(\zeta)^{-1}_s \ell(\zeta_s)\right|^2\,|\,X_s\right]\right]\\
&\leq 2\Vert \ell\Vert^2_{L^\infty}  \Eq\left[Z(\zeta)^2_s\Eq\left[\left(Z(\xi)^{-1}_s -Z(\zeta)^{-1}_s\right)^2 \Big|\,X_s\right]\right]\\
&\quad +2\Eq\left[Z(\zeta)^2_s\Eq\left[Z(\zeta)^{-2}_s \left|\ell(\xi_s)-\ell(\zeta_s)\right|^2\Big|\,X_s\right]\right]
\end{aligned}
\end{align}
where, by Lemma \ref{lem:ExpMartPrelim}-\textit{(i)},
\begin{align*}
\Eq\left[Z(\zeta)^2_s\Eq\left[\left(Z(\xi)^{-1}_s -Z(\zeta)^{-1}_s\right)^2 \Big|\,X_s\right]\right] & =\Eq\left[\Eq\left[Z(\zeta)^2_s\,|\,X_s\right]\left(Z(\xi)^{-1}_s -Z(\zeta)^{-1}_s\right)^2 \left|\ell(\xi_s)\right|^2\right]\\
& =\Eq\left[\Eq\left[Z(\zeta)^2_s\,|\,X_s\right]\frac{\left(Z(\xi)_s -Z(\zeta)_s\right)^2}{Z(\xi)^{2}_s Z(\zeta)^{2}_s} \right]\\&
\leq C\Eq\left[\frac{\left|Z(\xi)_s -Z(\zeta)_s\right|^2}{Z(\xi)^{2}_s}
\right]
\end{align*}
and
\begin{align*}
  \Eq\left[Z(\xi)^2_s\Eq\left[Z(\xi)^{-2}_s \left|\ell(\xi_s)-\ell(\zeta_s)\right|^2\,|\,X_s\right]\right]
  &= \Eq\left[\Eq\left[Z(\xi)^2_s\,|\,X_s\right]Z(\xi)^{-2}_s \left|\ell(\xi_s)-\ell(\zeta_s)\right|^2\right] \\
  &\leq C\Vert\ell\Vert^2_{Lip}\Eq\left[\left|\xi_s-\zeta_s\right|^2\right].
\end{align*}
Putting the two last upper bounds together,  we obtain the following bound for the l.h.s of \eqref{eq:aux_pivot1},
\begin{align}\label{eq:aux_pivot1_result}
\begin{aligned}
&\Eq\left[Z(\zeta)^2_s\left|\Eq\left[Z(\xi)^{-1}_s \ell(\xi_s)-Z(\zeta)^{-1}_s \ell(\zeta_s)\,|\,X_s\right]\right|^2\right]\\
&\leq C\Vert\ell\Vert^2_{L^{\infty}}\Eq\left[\frac{\left|Z(\xi)_s -Z(\zeta)_s\right|^2}{Z(\xi)^{2}_s}
\right]
  + C\Vert\ell\Vert^2_{Lip}\Eq\left[\left|\xi_s-\zeta_s\right|^2\right].
\end{aligned}
\end{align}

For the second integral in \eqref{proofst:conditionalexists_2}, again by Young's inequality, we have
\begin{align*}
&\Eq\left[\text{Trace}\left(\left(Z(\zeta)_s\Eq\left[Z(\zeta)^{-1}_s \gamma(\zeta_s)\,|\,X_s\right]-Z(\xi)_s\Eq\left[Z(\xi)^{-1}_s \gamma(\xi_s)\,|\,X_s\right]\right)\right.\right.\\
&\quad\quad\quad\times\left.\left.\left(Z(\zeta)_s\Eq\left[Z(\zeta)^{-1}_s \gamma(\zeta_s)\,|\,X_s\right]-Z(\xi)_s\Eq\left[Z(\xi)^{-1}_s \gamma(\xi_s)\,|\,X_s\right]\right)^t\right)\right]\\
&=\sum_{i,j,k=1}^d\Eq\left[\left(\left(Z(\zeta)_s\Eq\left[Z(\zeta)^{-1}_s \gamma^{i,j}(\zeta_s)\,|\,X_s\right]-Z(\xi)_s\Eq\left[Z(\xi)^{-1}_s \gamma^{i,j}(\xi_s)\,|\,X_s\right]\right)\right.\right.\\
&\quad\quad\quad\times\left.\left.\left(Z(\zeta)_s\Eq\left[Z(\zeta)^{-1}_s \gamma^{i,k}(\zeta_s)\,|\,X_s\right]-Z(\xi)_s\Eq\left[Z(\xi)^{-1}_s \gamma^{i,k}(\xi_s)\,|\,X_s\right]\right)\right)\right]\\
&\leq d\sum_{i,j=1}^d\Eq\left[\left|
Z(\zeta)_s\Eq\left[Z(\zeta)^{-1}_s \gamma^{i,j}(\zeta_s)\,|\,X_s\right]-Z(\xi)_s\Eq\left[Z(\xi)^{-1}_s \gamma^{i,j}(\xi_s)\,|\,X_s\right]
\right|^2\right].
\end{align*}
Each component of the above sum  can be bounded  in the same manner than \eqref{eq:aux_pivot1_result}, replacing $\Vert\ell\Vert$ by some $\Vert\gamma^{i,j}\Vert$.
Putting all together, we get, for some positive constant $C$,
\begin{align*}
\Eq\left[|Y(\zeta)_t - Y(\xi)_t|^2\right] \leq C \int_0^t
\Eq\left[\frac{\left|Z(\xi)_s -Z(\zeta)_s\right|^2}{Z(\xi)^{2}_s}+\frac{\left|Z(\xi)_s -Z(\zeta)_s\right|^2}{Z(\zeta)^{2}_s}
\right] ds
  + C \int_0^t \Eq\left[\left|\xi_s-\zeta_s\right|^2\right] ds.
\end{align*}
\end{proof}
Combining the result of Proposition {\ref{prop:LagContraction}} with Lemma \ref{lem:ExpMartPrelim}-$(ii)$, and following the same procedure as for \eqref{proofst:conditionalexists_0}-\eqref{proofst:conditionalexists_1}, we deduce with an appropriate choice of the constant $c$ that
\[
\Vert Y(\zeta) - Y(\xi)\Vert_{c}<\Vert \zeta - \xi\Vert_{c},\,\,\forall\,\zeta,\xi\in M^2(0,T).
\]
This ensures that the mapping $\zeta\in M^2(0,T)\mapsto Y(\zeta)\in M^2(0,T)$ which assigns to each element $\xi\in M^2(0,T)$, the solution $(Y(\zeta)_t;\,0\leq t\leq T)$ given by \eqref{eq:Transformedlinear} is contracting in $(M^2(0,T),\Vert~\Vert_c)$. This enable us to conclude on the existence and uniqueness of a strong  solution to \eqref{eq:TransformedNonlinear}.

By Girsanov transformation, this also enable us to conclude on the wellposedness of a weak  solution to \eqref{eq:NonlinearLangevin}.

\subsection{Strong uniqueness}\label{sec:strong}
The strong wellposedness of \eqref{eq:NonlinearLangevin} will be given by a direct application of the following theorem due to Veretennikov \cite{Veretennikov-80}:
\begin{theorem}[Theorem 1, \cite{Veretennikov-80}] Let $b:[0,\infty)\times\er^d\rightarrow\er^d$ be a bounded measurable function. Let $\sigma:[0,\infty)\times\er^d\rightarrow\er^d\times\er^d$ be such that
$a:(t,x)\mapsto a(t,x)=\sigma\sigma^*(t,x)$ is continuous, $x\mapsto a(t,x)$ is uniformly continuous in each compact  $K\subset\er^d$, for  any $t\in(0,T],\,0<T<\infty$, and for some positive $\lambda$
\begin{align*}
(\xi\cdot a(t,x)\xi) \geq \lambda|\xi|^2,\quad \text{for all }\xi\in\er^d,\,(t,x)\in[0,\infty)\times\er^d.
\end{align*}
Moreover, assume that there exist some Borel functions $\sigma_d\in W^{1,2d+2}_{\loc}(\er^n)$,  $\sigma_{d+1}\in L^{2d+2}_{\loc}((0,\infty);W^{1,2d+2}_{\loc}(\er^n))$ and $\sigma_L:(0,t)\times\er^n\times\er^n\rightarrow\er^d\times\er^d$, $\er^n\times\er^n$-Lipschitz continuous uniformly for $t\in(0,T]$, $0<T<\infty$, such that
\begin{align*}
\sigma(t,x)=\sigma_L(t,\sigma_d(x),\sigma_{d+1}(t,x)).
\end{align*}
Then, given a $\er^d$-valued standard Brownian motion $(w_t;\,t\geq 0)$, the stochastic differential equation:
$$
x_t=x+\int_0^t b(s,x_s)\,ds+\int_0^t \sigma(s,x_s)dw_s,\quad x\in\er^d,\,t\geq 0,
$$
has a unique strong solution.
\end{theorem}
We are going to show that the nonlinear diffusion coefficient
\[
(t,x)\mapsto \Ep[\gamma(Y_t)\,|\,X_t=x]
\]
is continuous and admits a derivative (in the Sobolev sense) w.r.t. $x$ such that $\nabla_x \Ep[\gamma(Y_t)\,|\,X_t=x]$ is locally in $L^{2d+2}$-integrable on $(0,T)\times\er^d$.
Before that, as a preliminary remark, let us point out that  owing to \hypcondo and \hypcondi, for all $t\geq 0$, the law of $X_t$ admits a density function $\rho_X(t,x)$. In addition, since $\rho_X(t,x)$ is a weak solution to
$$
\partial_t\rho_X+\nabla_x\left(\rho_X B_{\rho_X}\right)-\frac{1}{2}\trace\left(\nabla^2_x\times\left(\rho_X\sigma\sigma^*\right)\right)=0,
$$
where $B_{\rho_X}=B_{\rho_X}(t,x)$ is the bounded Borel measurable $\er^d$-vector field given by
$$
B_{\rho_X}(t,x)=\Ep\left[b(X_t,Y_t)\,|\,X_t=x\right],
$$
we have the following bounds (see e.g. Aronson \cite{Aronson-67})
\begin{equation}\label{eq:UpLowBounds}
c\int G^{1/\kappa}_{t}(x-x_0)\rho_X(0,x_0)\,dx_0\leq \rho_X(t,x)\leq C\int G^\kappa_{t}(x-x_0)\rho_X(0,x_0)\,dx_0,\quad x\in\er^d,\,0\leq t\leq T,
\end{equation}
where $G^{\kappa}_{t}$ is the centered Gaussian kernel with variance $\kappa^2 t$ and $\kappa,c,C$ are some finite positive constants depending only on $T,d$, $a_*$ and $a*$. Then, under the assumption \hypcondquatre, for all $0<R<\infty$, we have
\begin{align*}
  \inf_{x\in B(0,R)}\rho_X(t,x) &\geq  c\inf_{x\in B(0,R)}\int G^{1/\kappa}_{t}(x-x_0)\rho_X(0,x_0)\,dx_0 \\
  &\geq c\inf_{x\in B(0,R)}\int G^{1/\kappa}_{1}(x_0)\rho_X(0,x+tx_0)\,dx_0 \\
  &\geq c\inf_{x\in B(0,R)}\int_{B(0,R)} G^{1/\kappa}_{1}(x_0)\rho_X(0,x+tx_0)\,dx_0 \\
  &\geq c\inf_{x\in B(0,R),x_0\in B(0,R),0\leq t\leq T}\rho_X(0,x+tx_0)\left(\int_{B(0,R)} G^{1/\kappa}_{1}(x_0)\,dx_0\right),
\end{align*}
that leads to the following lower bound for $\rho_X$:
\begin{equation}\label{eq:lowerboundconditional}
\inf_{x\in B(0,R)}\rho_X(t,x) \geq m_R>0,\quad\text{with }\quad m_R:=c\inf_{z\in B(0,R+TR)}\rho_X(0,z)\left(\int_{B(0,R)} G^{1/\kappa}_{1}(x_0)\,dx_0\right).
\end{equation}
The positiveness of $\rho_X$ ensures that the component $\Ep\left[\gamma(Y_t)\,|\,X_t=x\right]$ is defined a.e. on $(0,T)\times\er^d$ and writes as a Borel measurable function:
$$
(t,x)\mapsto\Ep\left[\gamma(Y_t)\,|\,X_t=x\right]=\int \gamma(y)\mu^{Y\,|\,X=x}(t,dy,x),\,\text{for a.a. }(t,x)\in(0,T)\times\er^d,
$$
where $\int \gamma(y)\mu^{Y\,|\,X=x}(t,dy,x)$ is the disintegration of $\int \gamma(y)\mu(t,dx,dy)$ for $\mu(t)=\law(X_t,Y_t)$ w.r.t. $\rho_X(t,x)\,dx$. To exhibit the smoothness of $(t,x)\mapsto\Ep\left[\gamma(Y_t)\,|\,X_t=x\right]$, we prove below a general result showing that any distribution of the form
 $$
 (t,x)\mapsto \int m(y)\mu(t,dx,dy),\quad\text{for any bounded Borel function }m:\er^d\rightarrow \er,
 $$
 is absolutely continuous w.r.t. the Lebesgue measure on $\er^d$ and its related density is smooth in a suitable Sobolev sense. Such property is precisely given by the following lemma:

\begin{lemma}\label{lem:averaging} Assume that Assumptions \ref{hypo:weak} and \ref{hypo:strong} hold. For $(X_t,Y_t;\,0\leq t\leq T)$ the weak solution to \eqref{eq:NonlinearLangevin}, let $\mu(t)$ denote the joint law of $(X_t,Y_t)$ and let $\rho_X(t)$ be the density function of $\law(X_t)$ for $0\leq t\leq T$. Then, for any Borel measurable function $m:\er^d\rightarrow [0,\infty)$ not-identically equal to $0$, bounded, of class $\Cc^{2}(\er^d)$ on $\er^d$ such that its derivatives up to second order are bounded, the family of distributions
 $$
 \int m(y)\mu(t,dx,dy)
 $$
 admits a representant denoted $\overline{m\mu}$ in $L^p((0,T);W^{2,p}_{\loc}(\er^d))$ for any $p\geq d+2$.
\end{lemma}
Splitting $y\mapsto\gamma(y)$ into its positive part $(\gamma)^+$ and negative part $(\gamma)^-$, Lemma \ref{lem:averaging} ensures that, $\Ep\left[\gamma(Y_t)\,|\,X_t=x\right]$ rewrites according to
$$
\Ep\left[\gamma(Y_t)\,|\,X_t=x\right]=\frac{\overline{\gamma\mu}(t,x)}{\rho_X(t,x)}
\left(=\frac{\overline{(\gamma)^+\mu}(t,x)}{\rho_X(t,x)}+\frac{\overline{(\gamma)^-\mu}(t,x)}{\rho_X(t,x)}\right).
$$
Since Lemma \ref{lem:averaging}  also guarantees that $\rho_X$ and $\overline{\gamma\mu}$ are both in $L^p((0,T);W^{2,p}_{\loc}(\er^d))$ for any $p$ large enough, the lower bound \eqref{eq:lowerboundconditional} further ensures that
$(t,x)\mapsto\frac{\overline{\gamma\mu}(t,x)}{\rho_X(t,x)}$ is in $L^p((0,T);W^{2,p}_{\loc}(\er^d))$. We then conclude on the strong uniqueness of $(X_t,Y_t;\,0\leq t\leq T)$
solution to \eqref{eq:NonlinearLangevin}.

\begin{proof}[Proof of Lemma \ref{lem:averaging}] As a preliminary step, let us point out  that, since $\rho_X(0)$ is in $L^1(\er^d)\cap L^p(\er^d)$ for $p\geq 2d+2$, then $\rho_X(0)\in L^r(\er^d)$ for all $1\leq r\leq p$. Combined with the Gaussian upper-bound in \eqref{eq:UpLowBounds}, this estimate ensures that, whenever $g:\er^d\mapsto \er$ is bounded, for all $0\leq t\leq T$,
 $$
 \left|\int_{\er^d} \psi(x)\int_{\er^d} g(y)\mu(t,dx,dy)\right|\leq C\Vert g\Vert_{L^\infty(\er^d)}\Vert \psi\Vert_{L^q(\er^d)},\,\,1\leq \frac{q}{q-1}\leq 2d+2,\,\psi\in\Cc^\infty_c(\er^d).
 $$
 Riesz's representation theorem then implies that
 \begin{equation}\label{proofst:averaging_0}
 \int_{\er^d}  g(y)\mu(t,dx,dy)=\overline{g\mu}(t,x)\,dx,
 \end{equation}
 for some $\overline{g\mu}$ in $L^\infty((0,T);L^r(\er^d))$.

 By It\^o formula, for all $\phi\in\Cc^\infty_c((0,T)\times\er^d)$, we have
 \begin{align*}
 &\Ep\left[\int_0^T m(Y_t)\left(\partial_t\phi(t,X_t)+b(X_t,Y_t)\cdot \nabla_x \phi(t,X_t)+\frac{1}{2}\trace\left(
 \sigma\sigma^*(X_t)\nabla_x^2\phi(t,X_t)\right)\right)
 \,dt\right]\\
 &+\Ep\left[\int_0^T \phi(t,X_t)\left(\Ep\left[\ell(Y_t)\,|\,X_t\right]\cdot \nabla_y m(Y_t)+\frac{1}{2}\trace\left(
 \nabla_y^2m(Y_t)\left( \Ep\left[\gamma(Y_t)\,|\,X_t\right]
 \Ep\left[\gamma^*(Y_t)\,|\,X_t\right]\right)\right)\right)\,dt\right]=0.
 \end{align*}
 Rewriting the preceding expression into
 \begin{align*}
 &\int_{(0,T)\times\er^{2d}}m(y)\left(\partial_t\phi(t,x)+b(x,y)\cdot \nabla_x \phi(t,x)+\frac{1}{2}\trace\left( \sigma\sigma^*(x)\nabla_x^2\phi(t,x)\right)\right)\mu(t,dx,dy)\,dt\\
 &+\int_{(0,T)\times\er^{2d}}\phi(t,x)\left(\Ep\left[\ell(Y_t)\,|\,X_t=x\right]
 \cdot \nabla_y m(y)\right)\mu(t,dx,dy)\,dt\\
 &+\frac{1}{2}\int_{(0,T)\times\er^{2d}}\phi(t,x)\trace
 \left(\nabla_y^2m(y)\Ep\left[\gamma(Y_t)\,|\,X_t=x\right]\Ep\left[\gamma^*(Y_t)\,|\,X_t=x\right]\right)
 \mu(t,dx,dy)\,dt=0,
 \end{align*}
 we deduce that
 $$
 \overline{m\mu}(t,dx):=\int m(y)\mu(t,dx,dy),
 $$
 satisfies
 \begin{align*}
 &\int_{(0,T)\times\er^{d}}\left(\partial_t\phi(t,x)+\frac{1}{2}\trace\left( \sigma\sigma^*(x)\nabla_x^2\phi(t,x)\right)\right)\overline{m\mu}(t,dx)\,dt\\
 &+\int_{(0,T)\times\er^{d}}\left(\nabla_x\phi(t,x)\cdot\int m(y)b(x,y)\mu(t,dx,dy)\right)\,dt\\
 &+\int_{(0,T)\times\er^{d}}\left(\phi(t,x)\Ep\left[\ell(Y_t)\,|\,X_t=x\right]\cdot \int\nabla_ym(y)\mu(t,dx,dy) \right)\,dt\\
 &+\frac{1}{2}\int_{(0,T)\times\er^{2d}}\phi(t,x)\trace
 \left(\Ep\left[\gamma(Y_t)\,|\,X_t=x\right]\Ep\left[\gamma^*(Y_t)\,|\,X_t=x\right]\nabla^2_ym(y)\right)
 \mu(t,dx,dy)\,dt=0.
\end{align*}
Since $b$, $\ell$, $\gamma$, $m,\nabla_ym$ and $\nabla^2_ym$ are all bounded, we have
\begin{equation}\label{proofst:averaging_1}
\begin{aligned}
 &\int_{(0,T)\times\er^{d}}\left(\partial_t\phi(t,x)+\frac{1}{2}\mbox{Trace}\left( \sigma\sigma^*(x)\nabla_x^2\phi(t,x)\right)\right)\overline{m\mu}(t,dx)\,dt\\
 &+\int_{(0,T)\times\er^{d}}\left(E_1[t,x]+E_2[t,x]\right)\phi(t,x)\overline{m\mu}(t,dx)\,dt\\
 &=-\int_{(0,T)\times\er^{d}}\left(B[t,x]\cdot \nabla_x\phi(t,x)\right)\overline{m\mu}(t,dx)\,dt,
\end{aligned}
\end{equation}
where, following the convention in \eqref{proofst:averaging_0}, $B[t,x]=(B^{(i)}[t,x];1\leq i\leq d)$,  $E_1[t,x]$ and  $E_1[t,x]$ are given  by
\begin{align*}
B^{(i)}[t,x] &= \frac{\overline{mb^{(i)}\mu}}{\overline{m\mu}} (t,x)\ind_{\{\overline{m\mu}(t,x)\neq 0\}},\\
 E_1[t,x,\rho(t)]& =\sum_{i=1}^d\Ep\left[\ell^i(Y_t)\,|\,X_t=x\right]\frac{\overline{(\partial_{y_i}m)\mu}}{\overline{m\mu}}(t,x)\ind_{\{\overline{m\mu}(t,x)\neq 0\}},\\
 E_2[t,x] & =
 \sum_{i,j=1}^d\Big(\Ep\left[\gamma(Y_t)\,|\,X_t=x\right]\Ep\left[\gamma^*(Y_t)\,|\,X_t=x\right]\Big)^{i,j}\;
 \frac{\displaystyle\overline{(\partial^2_{y_iy_j}m)\mu}}{\displaystyle \overline{m\mu}}(t,x)\ind_{\{\overline{m\mu}(t,x)\neq 0\}}.
\end{align*}
Since $m$ is not identically equal to $0$ on $\er^d$, \eqref{eq:UpLowBounds} implies that $\overline{m\mu}(t,x)\neq 0$ a.e. on $(0,T)\times\er^d$ and that the fraction in each of the above functions are defined a.e. on $(0,T)\times\er^d$.) Noticing that $B$, $E_1$, $E_2$ are all locally bounded, the continuity of $\overline{m\mu}$ on $(0,T)\times\er^d$ and its local Sobolev regularity then follow from the application of the following results from Bogachev {\it et al.} \cite{BoKrRoSh-15}.

 \begin{proposition}[Corollaries 6.4.3 and 6.4.4 in  \cite{BoKrRoSh-15}]\label{prop:BoKrRoStEstimates}
   Let $\Dd\subset \er^d$ be an open set. Assume that $a:(0,T)\times\er^d\rightarrow\er^{d\times d}$ is uniformly elliptic such that $a$ and $a^{-1}$ are locally bounded on $(0,T)\times\er^d$ and that, for $p>d+2$,
\[\sup_{1\leq i,j\leq d}\sup_t\Vert a^{i,j}\Vert_{W^{1,p}(B(x_0,R))}<\infty\quad\text{for all  }x_0\in\Dd, ~0<R<\infty.\]
Assume that $b^1,b^2,\cdots,b^d$, $f^1,f^2,\cdots,f^d$ are in $L^p_{\loc}((0,T)\times\Dd)$, $c$ is in $L^{p/2}_{\loc}((0,T)\times\Dd)$, and assume that $\mu$ is a locally finite Borel measure on $(0,T)\times\er^d$ such that
   $$
   \int_{(0,T)\times\Dd}\left(\partial_t \phi+\text{Trace}(a\nabla^2\phi)+b\cdot\nabla\phi+c\phi\right)\,\mu(dt,dx)=\int_{(0,T)\times\Dd}f\cdot\nabla \phi,\,\,\forall\,\psi\in\Cc^\infty_c((0,T)\times\Dd).
   $$
   Then $\mu$ has a locally H\"older continuous density that belongs to the space $L^p(J;W^{1,p}(V))$ for all $J\subset (0,T),V\subset \Dd$ such that $J\times V$ has compact closure in $(0,T)\times\Dd$.
 \end{proposition}
 \end{proof}

\medskip

\noindent
 {\textbf{Acknowledgement}: The second author has been supported by the Russian Academic Excellence Project "$5$-$100$".}

\appendix
\section{Appendix}

The proof of the estimate \eqref{eq:StroockVaradhan-estim} relies on the original arguments exhibited  in \cite{StrVar-79}  (pages 304--306) in a one-dimensional setting.
 We simply extend the result  to a multidimensional case: For $\delta>0$, define
$$
\beta_\delta(t)=\ind_{\{\delta \leq t\leq 1/\delta\}},\,\, G^{\gamma,\delta}_{t}(x,y)=\beta_\delta(t)G^{\gamma}_{t}(x,y).
$$
Let $f$ be a $\Cc^\infty_c((0,T)\times\er^d)$-function so that $(t,x)\mapsto \int_0^t \triangle_xG^{\gamma}_{t-s}(f)(s,x)\,ds=\int_0^t \triangle_xG^{\gamma}_{t-s}{*} f(s,x)\,ds$ (for simplicity $G^{\gamma}_{t}$ {denotes} the $\Nn(0,\gamma^2 t)$-Gaussian density function{/kernel}) is $\Cc^\infty((0,T)\times\er^d)\cap L^2((0,T)\times\er^d)$. By Parseval's equality: $\Vert h\Vert_{L^2(\er^m)}=\frac{1}{(2\pi)^m}\Vert \Ff(h)\Vert_{L^2(\er^m)},\,h\in L^2(\er^m),\,m\geq 1$, we get
$$
\big\Vert \int_0^{.} \triangle_xG^{\gamma,\delta}_{.-s}{*} f(s,x)\,ds\big\Vert_{L^2((0,T)\times\er^d)}=\frac{1}{(2\pi)^{d+1}}\big\Vert
\Ff(\triangle_xG^{\gamma,\delta})\Ff(f\ind_{\{[0,T]\}})\big\Vert_{L^2(\er^{d+1})}
$$
where $\Ff$ denote the Fourier transformation along the variables $t$ and $x$:
$$
\Ff(f\ind_{\{[0,T]\}})(\tau,\xi)=\int_\er\int_{\er^d} \exp\{-\mathbf{i}t\tau-\mathbf{i}x\cdot \xi \}f(t,x)\ind_{[0,T]}\,dt\,dx.
$$
Since
$$
\Ff(\triangle_xG^{\gamma,\delta})(\tau,\xi)=-|\xi|^2\Ff(G^{\gamma,\delta})(\tau,\xi)
$$
with
\begin{align*}
\Ff(G^{\gamma,\delta})(\tau,\xi)&=\int_0^T\beta_{\delta}(t) \exp\{-\mathbf{i}t\tau\}\left(\int_{\er^d}\exp\{-\mathbf{i}\xi\cdot x\} G^{\gamma}(t,x)\,dx\right)\,dt\\
&=\int_0^T\beta_{\delta}(t) \exp\{-\mathbf{i}t\tau\}\exp\{-t\gamma^2|\xi|^2/2\}\,dt\\
&=\frac{2}{2\mathbf{i}\tau-\gamma^2|\xi|^2}\left(\exp\{-\delta(\mathbf{i}\tau-\gamma^2|\xi|^2/2)\}-
\exp\{-\frac{(\mathbf{i}\tau-\gamma^2|\xi|^2/2)}{\delta}\}\right),
\end{align*}
we get
$$
|\xi|^2\left|\Ff(\triangle_xG^{\gamma,\delta})(\tau,\xi)\Ff(f\ind_{[0,T]})(\tau,\xi)\right|\leq \left|\frac{-2|\xi|^2}{2\mathbf{i}\tau-\gamma^2|\xi|^2}\right||\Ff(f\ind_{[0,T]})|(\tau,\xi)\leq \frac{2}{\gamma^2}|\Ff(f\ind_{[0,T]})|(\tau,\xi).
$$
Integrating both sides of the preceding inequality over $\er^{d+1}$, it follows that
\begin{align*}
  \frac{1}{(2\pi)^{d+1}}\big\Vert \int_0^{.} \Ff(\triangle_xG^{\gamma,\delta})\Ff(f\ind_{[0,T]})\big\Vert_{L^2((0,T)\times\er^d)}&\leq \frac{2}{\gamma^2(2\pi)^{d+1}}\big\Vert \Ff(f\ind_{[0,T]})\big\Vert_{L^{2}((0,T)\times\er^d)}\\
  &\leq \frac{2}{\gamma^2}\big\Vert f\big\Vert_{L^{2}((0,T)\times\er^d)}.
\end{align*}
 Since $\int_0^{.}\triangle_x G^\gamma_{.-s}(f(s))\,ds=\lim_{\delta\rightarrow 0^+}  \int_0^{.}\triangle_x G^{\gamma,\delta}_{.-s}(f(s))\,ds $, we deduce, by extension, that
\[
\big\Vert\int_0^{.}\triangle_x G^\gamma_{.-s}(f(s))\,ds \big\Vert_{L^2((0,T)\times\er^d)}\leq \frac{2}{\gamma^2}\big\Vert f\big\Vert_{L^2((0,T)\times\er^d)},\,\forall f\in \Cc^\infty_c((0,T)\times\er^d).
\]
Since $\Cc^\infty_c((0,T)\times\er^d)$ is dense in $L^2((0,T)\times\er^d)$, we conclude \eqref{eq:StroockVaradhan-estim}.

In the same way, for any given positive definite $\er^{d\times d}$-matrix $\Gamma$ satisfying $\xi\cdot \Gamma \xi\geq \gamma^2|\xi|^2$ for some $\gamma\neq 0$, let $G^\Gamma_t$ be the $\Nn(0,\Gamma t)$-Gaussian density function and define, for $\beta_\delta$ as above,
\[
G^{\Gamma,\delta}_{t}=\beta_\delta(t)G^{\Gamma}_{t}.
\]
Observing that
\begin{align*}
\left|\Ff(\partial^2_{x_ix_j} G^{\Gamma,\delta})(\tau,\xi)\right|&=\left|\xi_i\xi_j\right|
\left|\Ff(G^{\Gamma,\delta})(\tau,\xi)\right|\\
&=\left|\xi_i\xi_j\right|\left|\int_0^T\beta_{\delta}(t) \exp\{-\mathbf{i}t\tau\}\exp\{-t(\xi \cdot \Gamma\xi )/2 \}\,dt\right|\leq \frac{2\left|\xi_i\xi_j\right|}{\left|2\mathbf{i}\tau- (\xi \cdot \Gamma\xi )\right|}\leq \frac{2}{\gamma^2},
\end{align*}
then, for any family $F^{i,j}\in \Cc^\infty_c((0,{T})\times\er^d),\,1\leq i,j\leq d$, we have
$$
|\xi|^2\left|\Ff(\partial^2_{x_ix_j} G^{\Gamma,\delta})(\tau,\xi)\Ff(F^{i,j}\ind_{[0,T]})(\tau,\xi)\right|\leq \frac{2}{\gamma^2}|\Ff(F^{i,j}\ind_{[0,T]})|(\tau,\xi)
$$
from which we deduce, as previously, that
\begin{align*}
 \big\Vert \sum_{i,j=1}^d \int_0^{.}\partial^2_{x_ix_j} G^\Gamma_{.-s}(F^{i,j}(s))\,ds\big\Vert_{L^2((0,{T})\times\er^d)}
 &=\big\Vert \lim_{\delta\rightarrow 0^+}\sum_{i,j=1}^d \int_0^{.}\partial^2_{x_ix_j} G^{\Gamma,\delta}_{.-s}(F^{i,j}(s))\,ds\big\Vert_{L^2((0,{T})\times\er^d)}\\
 &\leq \frac{2}{\gamma^2}\sum_{i,j=1}^d \Vert F^{i,j} \Vert_{L^2((0,{T})\times\er^d)}.
\end{align*}


\end{document}